\documentclass[draft]{amsart}
\usepackage{latexsym}
\usepackage{amssymb, amscd, amsfonts}
\usepackage[mathscr]{eucal}
\usepackage{graphicx}
\usepackage{amscd}
\usepackage[all]{xy}%\SelectTips{cm}{}\CompileMatrices

\newtheorem{theorem}{Theorem}[section]
\newtheorem{lemma}[theorem]{Lemma}
\newtheorem{corollary}[theorem]{Corollary}
\newtheorem{proposition}[theorem]{Proposition}

\theoremstyle{definition}
\newtheorem{remark}{Remark}
\newtheorem{definition}[theorem]{Definition}
\newtheorem{example}[theorem]{Example}

\DeclareMathOperator{\McN}{\mathscr{M}}
\DeclareMathOperator{\Pol}{\mathscr{P}}

\DeclareMathOperator{\FP}{\mathsf{U}_{fp}}
\DeclareMathOperator{\Set}{\mathsf{SET}}
\DeclareMathOperator{\wasc}{\mathsf{W}}
\DeclareMathOperator{\QP}{\mathsf{P}_{\Zed}}

\DeclareMathOperator{\McNn}{\mathscr{M}{([0,1]^{\it n})}}
\DeclareMathOperator{\cube}{[0,1]^{\it n}}
\DeclareMathOperator{\R}{\mathbb{R}}
\newcommand\Zed{{\ensuremath{\mathbb{Z}}}}

\newcommand{\restrict}[1]{\negthinspace \upharpoonright \negthinspace #1}

\DeclareMathOperator{\conv}{\rm conv}
\DeclareMathOperator{\den}{\rm den}
\DeclareMathOperator{\ver}{\rm vert}
\DeclareMathOperator{\dist}{\rm dist}

\title[Simplicial geometry of unital $\ell$-groups]
{Simplicial geometry of unital lattice-ordered abelian groups
}

\author[L.M.Cabrer]{Leonardo Manuel Cabrer %\footnote{}
}
\address[L.M.Cabrer]{Dep. de Matem{\'a}ticas -- Facultad de Ciencias Exactas \\
Universidad Nacional del Centro de la Provincia de Buenos Aires \\
Pinto 399 -- Tandil (7000) \\
Argentina }
\email{lcabrer@exa.unicen.edu.ar }

\date{\today. The main results of this paper have been obtained
while the author was holding a postdoctoral position  in the
Mathematiches Institute of the University of Bern.}

\begin{document}

\begin{abstract}
By an $\ell$-group $G$ we mean a
lattice-ordered abelian group. This paper is concerned
with the category  $\FP$  of finitely presented
 {\it unital} $\ell$-groups,  those
$\ell$-groups having a distinguished order-unit $u$.
Using the duality between $\FP$  the
category of rational polyhedra, we will provide (i)  a
construction of finite limits and co-limits in $\FP$; (ii)  a
Cantor-Bernstein-Schr\"oder theorem for finitely presented unital $\ell$-groups;
(iii)   a geometrical characterization of finitely generated
subalgebras of free objects of $\FP$.
\end{abstract}

\maketitle

%%%%%% %%%%%% %%%%%% %%%%%% %%%%%% %%%%%% %%%%%% %%%%%% %%%%%% %%%%%%
%%%%%% %%%%%% %%%%%% %%%%%% %%%%%% %%%%%% %%%%%% %%%%%% %%%%%% %%%%%%
%%%%%% %%%%%% %%%%%% %%%%%% %%%%%% %%%%%% %%%%%% %%%%%% %%%%%% %%%%%%
%%%%%% %%%%%% %%%%%% %%%%%% %%%%%% %%%%%% %%%%%% %%%%%% %%%%%% %%%%%%
{\bf MSC2010}\ \ {Primary:   06F20, 52B20, 18B30.
Secondary:  05E45,  52B11, 18A35, 55U05, 55U10
 57Q05.}

\section{Introduction}

A {\it unital $\ell$-group} $(G,u)$ is
an abelian group $G$ equipped with a translation invariant lattice-order
and with a distinguished {\it order unit}, i.e.
an element $0\leq u\in G$ whose positive integer multiples eventually
dominate every element of $G$.
A {\it unital $\ell$-homomorphism} between unital $\ell$-groups is
a group homomorphism that preserves the order unit and the lattice structure.

%%%%%%%%%%%%%% Definition of Finitely presented
As a particular case of a general definition \cite[p. 286]{BS1981},
  a unital $\ell$-group $(G,u)$ is said to be
   {\it finitely presented} if there exists a finite set
    $\{g_1,\ldots,g_n\}\subseteq G$ along with a finite set of equations
    $s_1=t_1,\ldots,s_m=t_m$ in the language of
    unital $\ell$-groups with $n$-variables such that

\begin{enumerate}
   \item $s_i(g_1,\ldots,g_n)=t_i(g_1,\ldots,g_n)$ for each $i=1,\ldots,m$ and
   \item if $(H,v)$ is a unital $\ell$-group and $h_1,\ldots,h_n\in H$ satisfy
   $s_i(h_1,\ldots,h_n)=t_i(h_1,\ldots,h_n)$ for each $i=1,\ldots,m$,
   then there exists a unique unital $\ell$-homomorphism
$h\colon (G,u)\rightarrow(H,v)$
   such that $h(g_1)=h_1,\ldots, h(g_n)=h_n.$
\end{enumerate}
We denote by $\FP$
%
%\begin{quote}
%{\bf D-to-L:  here are the possible correct
%usages of the verb ``denote'':
%\begin{itemize}
%\item we denote by $\partial X$ the boundary of $X$
%\item we let $\partial X$ denote the boundary of $X$
%\item the boundary of $X$ will be denoted $\partial X$
%\item $\partial X$ will denote the boundary of $X$
%\end{itemize}
%}
%\end{quote}
the category of finitely presented unital
$\ell$-groups with unital $\ell$-homomorphisms.

For $n=1,2,\ldots$ we let $\McNn$ denote the unital $\ell$-group of
all continuous functions $f\colon [0,1]^{n}\rightarrow\mathbb R$
having the following property:
there are linear polynomials $p_{1},\ldots,p_{m}$ with integer coefficients
such that for all $x\in [0,1]^{n}$ there is $i\in \{1,\ldots,m\}$
with $f(x)=p_{i}(x)$.
$\McNn$ is equipped with the pointwise operations $+,-,\max,\min$ of
$\mathbb R$,
and with the constant function $1$ as the distinguished order unit.

$\McNn$  is a free object in the category of unital $\ell$-groups,
in the following sense:

\begin{proposition}\label{proposition:free}{\rm (\cite[Corollary
4.16]{Mu1986})}
   The coordinate maps $\xi_ {i} \colon [0,1]^{n}\to \mathbb R$ together with
    the order unit $1$ form  a generating set of $\McNn$.
   For every unital $\ell$-group $(G,u)$ and $0\leq g_{1},\ldots,g_{n}\leq u$,
   if the set $\{g_{1},\ldots,g_{n}, u\}$ generates $G$  then
   there is a unique unital $\ell$-homomorphism $\psi$ of $\McNn$ onto $G$
   such that $\psi(\xi_ {i})=g_{i}$ for each $i=1,\ldots,n.$
\end{proposition}

  An {\it ideal} $\mathfrak i$ of a unital $\ell$-group $(G,u)$
  is the kernel of a unital $\ell$-homomorphism  of $(G,u)$,
(\cite[p.8  and  1.14]{Go1986}).
% %
  $\mathfrak i$ is {\it principal} if it is singly (=finitely) generated.

As a
  consequence of Proposition \ref{proposition:free},
   a unital $\ell$-group $(G,u)$ is finitely presented
iff for some $n=1,2,\ldots,$ $(G,u)$
is isomorphic to the quotient of $\McNn$ by
some principal ideal $\mathfrak j$,
in symbols, $(G,u)\cong \McNn/\mathfrak j$.

The characterization of finitely presented unital $\ell$-groups presented in
Proposition \ref{proposition:free} relies
 on the free objects $\McNn$ and their universal property. In \cite{CM2011} an intrinsic characterization of finitely presented unital $\ell$-groups is
 given in terms of  special sets of generators, called  bases. The notion of
{basis}
was introduced in \cite{MM2007} as a purely algebraic counterpart of
Schauder bases.
In  \cite[Theorem 4.5]{MM2007}
it is proved that if a
unital $\ell$-group $(G,u)$  is isomorphic to an $\ell$-group of
real-valued functions defined on some set $X$ (that is,
$G$  is {\emph{Archimedean}}) then it is finitely presented
iff it has a basis. In \cite[Theorem 3.1]{CM2011} it is proved that the Archimedean assumption
can be dropped: thus,  $(G,u)$ is finitely presented iff it has a basis.

In Section 2  we give a detailed account of
the main tools used in this paper, namely
  the categorical duality
between finitely presented unital $\ell$-groups
and rational polyhedra, and the combinatorial
representation of rational polyhedra as  weighted
abstract simplicial complexes.

  Section 3 is devoted to the construction
  of limits and co-limits in these two categories.

Finally, in Section 4,  all the machinery of the earlier chapters
will be combined with the algebraic-topological analysis of projective
unital $\ell$-groups  in \cite{CM20XX}
to give geometric and algebraic characterizations of finitely generated subalgebras of
the free unital $\ell$-groups  $\McNn$.

%%%%%%%%%%%%%%%%%%%%%%%%%%%%%%%%%%%%%%%%%%%%%%%%%%%%%%%%%%%%%%%%%%%%%%%%%%%%%%%%%%%%%%%%%%%%%%%%%%
%%%%%%%%%%%%%%%%%%%%%%%%%%%%%%%%%%%%%%%%%%%%%%%%%%%%%%%%%%%%%%%%%%%%%%%%%%%%%%%%%%%%%%%%%%%%%%%%%%
\section{Preliminaries}

%%%%%%%%%%%%%%%%%%%%%%%%%%%%%%%%%%%%%%%%%%%%%%%%%%%%%%%%%
\subsection{Regular triangulations}
We refer to \cite{Ew1996}, \cite{Gl1970} and \cite{St1967} for
background in elementary polyhedral topology and simplicial complexes.

%%%%%%%%%%%%%%%%%% Definition of vertices and faces of a simplex
For any simplex $S$ we denote by $\ver(S)$ the set of its vertices.
For any $F\subseteq \ver(S)$, the convex hull $\conv(F)$ is called a
{\it face} of $S$.
A {\it polyhedron} $P$  in $\R^{n}$  is a finite union of (always
closed) simplexes $P=S_{1}\cup\cdots\cup S_{t}$ in $\R^n$.

%%%%%%%%%%%%%%%%%%%%%%%%% Definition of Polyhedron
A simplex $S$ is said to be {\it rational} if the coordinates of each
$v\in \ver(S)$ are rational numbers.
$P$ is said to be a {\it rational polyhedron} if there are rational
simplexes $T_{1},\ldots,T_{l}$
such that $P=T_{1}\cup\cdots\cup T_{l}$.
%

%%%%%%%%%%%%%%%%%% Definition of support and rational simplicial complex
For every simplicial complex  ${\Delta}$, its {\it support}
$|\Delta|$ is the pointset union of all simplexes of $\Delta$,
and $\ver(\Delta)$ is the set of its vertices, i.e.
 the set of the
vertices of its simplexes.
We say that the simplicial complex  $\Delta$ is {\it rational} if all
simplexes of $\Delta$ are rational.
%%%%%%%%%%%%%%%%%%%%%%%%%%%%%%%%%%%%%%%%%%%

%%%%%%%%%%%%%%%%%%%%%%%%%%%%%% Triangulation
Given a rational polyhedron $P$,   a {\it triangulation}  of
$P$ is a rational simplicial complex $\Delta$ such that $P=|\Delta|$.
In \cite[Theorem 1]{Be1977} it is proved that
$\Delta$ exists for every rational polyhedron $P$.
%%%%%%%%%%%%%%%%%%%%%%%%%%%%%%%%%%%

%%%%%%%%%%%%%%%%%%%%%%%%%%% Rationality from now on
In the rest of this paper every simplex, polyhedron, and simplicial
complex will be rational.
Accordingly, the adjective ``rational'' will be omitted unless it is
strictly necessary.
%%%%%%%%%%%%%%%%%%%%%%%%%%%%%%%%%%%%%%%%%%

%%%%%%%%%%%%%%% Definition denominator
For $v$ a rational point in $\R^{n}$ we let $\den(v)$ denote the
least common denominator of the coordinates of $v$.
The vector $\widetilde{v}=\den(v)(v,1)\in\Zed^{n+1}$ is called the
{\it homogeneous correspondent} of $v$.
%%%%%%%%%%%%%%%%%%%%%%%%%%
%
%%%%%%%%%%%%%% Definition of regular simplex
A simplex $S$ is called {\it regular} if the set of homogeneous
correspondents of its vertices is part of a basis of the free abelian
group $\Zed^{n+1}.$
%%%%%%%%%%%%%%%%%%%%%%%%%
%
%%%%%%%%%%%%%%%% Definition Regular triangulation
By a {\it regular triangulation} of a polyhedron $P$ we understand a
triangulation of $P$ consisting of regular simplexes.
%%%%%%%%%%%%%%%%%

The following proposition  was proved in \cite[Theorem 1.2]{Mu1986}
under the  assumption that $X\subseteq [0,1]^{n}$. However,
it is easy to see  that  the proof is the same for
all $X\subseteq \mathbb R^n$ (also see
%\cite[Lemma 2.3]{MS201X} and
  \cite[Lemma 3.1]{MuXXXX}):

\begin{proposition}\label{proposition:poly}
For any
   set   $X\subseteq \R^{n}$ the following statements are equivalent:
   \begin{enumerate}
     \item $X$ coincides with the support of some regular complex $\Delta$;
     \item $X$ is a rational polyhedron.
   \end{enumerate}
\end{proposition}

%%%%%%%%%%%%%%%%%%%%%%%%%%%%%%%%%%%%%%%%%%%%%%%%%%%%%%
\subsection{Farey subdivisions}
%%%%%%%%%%%%%

%%%%%%%%%%%%%%%%%%% Definition of subdivision
Given a polyhedron $P$ and triangulations $\Delta$ and $\Sigma$ of $P$
we say that $\Delta$ is a {\it subdivision} of $\Sigma$ if every
simplex of $\Delta$ is contained in a simplex of $\Sigma$.
%%%%%%%%%%%%%%%%%%%%%%%%%%%%%%
%
%%%%%%%%%%%%%%%%% Definition of Blow-up
For any  point $p \in P$,
the {\it blow-up $\Delta_{(p)}$ of $\Delta$ at $p$} is the
subdivision of $\Delta$
given by replacing every simplex $S\in \Delta$ that contains $p$
by the set of all simplexes of the form $\conv(F\cup\{p\})$,
where $F$ is any face of $S$ that does not contain $p$
(see \cite[p. 376]{Wl1997} or \cite[III, Definition 2.1]{Ew1996},  where
blow-ups are called stellar subdivisions).
%%%%%%%%%%%%%%%%%%%%%%%%%%%%

%%%%%%%%%%%%%%%%%%%%%Definition Farey Mediant
For any regular  $m$-simplex $S =\conv(v_{0},\dots,v_m) \subseteq
\mathbb R^{n}$,
the {\it Farey mediant} of (the vertices of) $S$
is the rational point $v$ of $S$ whose homogeneous correspondent
$\tilde v$ equals $\widetilde{{v}_0}+\cdots+\widetilde{{v}_m}$.
%%%%%%%%%%%%%%%%%%%%%%%%
%
%%%%%%%%%%%%%%%%%%%% Definition Fare Blow-up (down)
If $S$ belongs to a triangulation $\Delta$ and $v$ is the Farey mediant of $S$
then the blow-up $\Delta_{(v)}$ is a regular triangulation iff
so is $\Delta\,\,\,$     (\cite[V, 6.2]{Ew1996}).
$\Delta_{(v)}$ will be called the {\it Farey blow-up} of $\Delta$ at $v$.
By a {\it (Farey) blow-down} we understand  the inverse  of
a (Farey) blow-up.
%%%%%%%%%%%%%%%%%%%%%%%%%%
\smallskip

The proof of the ``weak Oda conjecture'' by
  Morelli \cite{Mo1996}
and  W\l odarczyk
\cite{Wl1997}  yields:
\begin{lemma}\label{lemma:weakoda}
   Let $P$ be a polyhedron.
   Then any two regular triangulations of $P$ are connected by a
finite path of Farey blow-ups and Farey blow-downs.
\end{lemma}

For later use in this paper,  we
  recall here some properties of regular triangulations.
\begin{lemma}
\label{Lem_Triang-Subset}
   Let $P\subseteq Q\subseteq\R^{n}$ be rational polyhedra and
   $\Delta$ be a regular triangulation of $P$.
   Then there exists a regular triangulation $\Delta_Q$ of $Q$ such that
   the set
   $\Delta_{P}=\{S\in\Delta_Q\mid S\subseteq P\}$ is a subdivision of
   $\Delta$.
   Moreover, $\Delta_Q$ can be so chosen that $\Delta_P$ is a
   { \rm full}
   subcomplex of $\Delta_Q$,  in the sense that
    $\Delta_P=\{S\in\Delta_{Q}\mid \ver(S)\subseteq P\}$.
\end{lemma}
\begin{proof}
   Let $\nabla_0$  be a rational triangulation of $Q$.
  From \cite[Addendum 2.12]{RS1972} we obtain
a triangulation  $\nabla_1$ of $Q$ which is
a  subdivision   of $\nabla_0$ and also satisfies
$S=\bigcup\{T\in \nabla_1 \mid T\subseteq S\}$,
  for each $S\in\Delta$.
  By \cite[Lemma 3.4]{RS1972}, there is a
  subdivision $\nabla_2$ of $\nabla_1$ such
  that $\{T\in \nabla_2\mid T\subset P\}$
  is a full subcomplex of $\nabla_2$. By \cite[Corollary, p. 242]{Be1977},
    there is no loss of generality
to assume that $\nabla_2$ is rational.

   The  desingularization process described in
    \cite[Chapter 9]{CDM2000}
   then yields
    a regular triangulation $\nabla_3$ which is a subdivision of
    $\nabla_2$.  Since $\nabla_3$ is obtained form $\nabla_2$ by blow-ups,
    then   $\{T\in \nabla_3\mid T\subset P\}$ is a
    regular triangulation of $P$ which is also
    a  subdivision of $\Delta$ and a
    full subcomplex of $\nabla_3$.
\end{proof}

\begin{lemma}\label{Lem_Sub-Triang-Subset}
   Let $P\subseteq Q\subseteq\R^{n}$ be rational polyhedra,
   and $\Delta_P$ and $\Delta_Q$  be  regular triangulations of $P$ and $Q$
   such that $\Delta_P\subseteq \Delta_Q$.
   If $\nabla_P$ is a regular subdivision of $\Delta_P$
   then there exists a regular triangulation $\nabla_Q$  of $Q$ such that
  $\nabla_P\subseteq \nabla_Q$
  and $\nabla_Q$   is a subdivision of $\Delta_Q$.
\end{lemma}
\begin{proof}
Let $K=\{S\in \Delta_Q\mid \ver(S)\subseteq P\mbox{ and }S\nsubseteq P\}$ and $S$ be a maximal element in $K$.
Then the Farey blow up $(\Delta_Q)_{(v)}=\nabla_1$, where $v$ is the Farey mediant of $S$, is a regular triangulation of $Q$ such that $\Delta_P\subseteq \nabla_1$. By the maximality of $S$ in $K$, $K_1=\{S\in \nabla_1\mid \ver(S)\subseteq P\mbox{ and }S\nsubseteq P\}=K\setminus\{S\}$. Repeating this process, we obtain a sequence of regular complexes $\nabla_1, \ldots, \nabla_{r}$ and a sequence of sets $K=K_0\supseteq K_1\supseteq\cdots\supseteq K_{r}$ where each $\nabla_{k+1}$ is obtained by blowing-up $\nabla_{k}$ at the Farey mediant of some maximal $S$ in $K_{k}$ and $K_{k+1}=K_{k}\setminus S$.
Since $K$ is finite this process terminates at some $t$. By construction,
 $\Delta_P$ is a full subcomplex of $\nabla_t$  and $\nabla_t$ is a subdivision of $\Delta_Q$.

For each $S\in \nabla_t$ we define:
$$
\ver_P(S)=\ver(S)\cap P\ \ \mbox{, }\ \ \ver_Q(S)=\ver(S)\setminus \ver_P(S),\mbox{  and}
$$
$$U_S=\{\conv(\ver_Q(S)\cup T)\mid T\in \nabla_P\mbox{ and }T\subseteq S\cap P\}.$$

Since $\Delta_P$ is a full subcomplex of $\nabla_t$ and $\nabla_Q$ is a subdivision of $\Delta_Q$, it follows that
$\nabla_0=\bigcup\{U_S\mid S\in \nabla_t\}$ is a simplicial complex and a triangulation of $Q$.

An application to $\nabla_0$  of the desingularization procedure
of \cite[1.2]{Mu1988}  yields a regular triangulation
$\nabla_Q$. $\nabla_Q$ is obtained from $\nabla_0$ by a suitable sequence of blow ups at non-regular simplexes. Since $\nabla_P\subseteq \nabla_0$ is regular, none of its simplexes is modified by the application of this procedure. Therefore, $\nabla_P\subseteq\nabla_Q$ and $\nabla_Q$ is a subdivision of $\Delta_Q$, as desired.
\end{proof}

%%%%%%%%%%%%%%%%%%%%%%%%%%%%%%%%%%%%%%%%%%%%%%%%%%%%%%
\subsection{The category of rational polyhedra}
%%%%%%%%%%%%%%%%%%%%%
\begin{definition}(\cite[Definition 3.1]{Mu2011})
   Given a rational polyhedra  $P\subseteq \R^{n}$ a map $\eta\colon P\rightarrow \R^{m}$ is called a {\it \Zed-map}
   if there is a  triangulation $\Delta$  of $P$
   such that over every simplex $T$  of  $\Delta$, $\eta$ coincides
with an affine linear map $\eta_{T}$ with integer coefficients.
\end{definition}

The following lemmas are easy consequences of the definition. For detailed proofs in the case of rational polyhedra contained in some $n$-cube $[0,1]^{n}$ see \cite[\S 3]{Mu2011}.

\begin{lemma}\label{Lem_ImagesZmorph0}
Suppose   $P\subseteq \R^{n}$ and $\eta\colon P\rightarrow R^{m}$ is a \Zed-map.
Then $\eta(P)$ is a polyhedron in $\R^m$.
\end{lemma}
%\begin{proof}
  % Let $\Delta$  be a triangulation of $P$ such that $\eta$ is
%linear on each simplex of  $\Delta$.
%   %
%   Let $T$ be a simplex of $\Delta$, and let $\{v_1,\ldots,v_r\}$ be
%the vertices of $T$.
%   %
%   Then
%    $\eta(T)=\conv(\eta(v_1),\ldots,\eta(v_r))$,
%   and $\eta(T)$ is the union of simplexes whose vertices are
%contained in $\{\eta(v_1),\ldots,\eta(v_r)\}$.
%   %
%   Since all points
%   $\{v_1,\ldots,v_r\}$ are rational and $\eta$ is linear on $T$ with
%integer coefficients,
%   then
%   all points  $\{\eta(v_1),\ldots,\eta(v_r)\}$ are rational.
%   %
%This shows that $\eta(T)$ is a rational polyhedron.
%   %
%The desired conclusion now follows
%  from the fact that $\eta(P)=\bigcup\{\eta(T)\mid T\in\Delta\}$.
%   \medskip
%
%\end{proof}

\begin{lemma}\label{Lem_StableTriangul}
   Let $\eta\colon P\rightarrow Q$ be a \Zed-map and
   $\Delta$  a regular triangulation of $P$.
   Then there exists a regular triangulation  $\nabla$ of $P$
   such that $\nabla$ is a subdivision of $\Delta$ and $\eta$ is
linear over each simplex in $\nabla$.
\end{lemma}

\begin{lemma}\label{Lem_CarZed}
   Given polyhedra $P\subseteq \R^{n}$, $Q\subseteq \R^{m}$ and
    a map $\eta\colon P\rightarrow Q$, let
   $\xi_1,\ldots,\xi_m\colon Q\rightarrow \R$ denote the coordinate maps.
Then
   $\eta$ is a $\Zed$-map iff $\xi_i\circ\eta:P\rightarrow
\xi_i(P)$ is a \Zed-map for each $i=1,\ldots,m$.
\end{lemma}

\begin{lemma}\label{Lem_ZeroP}
    Given a rational polyhedron  $P\subseteq \R^{n}$. A subset $P\subseteq Q$ is a rational polyhedron
    iff
    there exists a \Zed-map $f\colon Q\rightarrow \R$ such
that $P=f^{-1}(0)$.
\end{lemma}
\begin{proof}
  In \cite[Propositions 5.1 and 5.2]{MM2007} and  \cite[Lemma 3.2]{MuXXXX}, the result was proven for the case $P\subseteq [0,1]^n$. The same argument replacing $[0,1]^n$ by an arbitrary rational polyhedron $Q$ proves the result of this lemma.
\end{proof}

We denote by $\QP$  the category
  whose objects are rational polyhedra in
$\mathbb R^n$\,\,\,({\rm for}\,\, $n=1,2,\ldots),$
and whose arrows are $\Zed$-maps.

Following  \cite[Definition 4.4]{MuXXXX} and \cite[Definition 3.9]{Mu2011},   a map $\eta\colon P\rightarrow Q$ is a {\it  \Zed-homeomorphism}  if
   it is a one-one  \Zed-map of $P$  onto $Q$ and its inverse
$\eta^{-1}$ is also a \Zed-map.
Thus  {$\Zed$-homeomorphisms} are the same
as   iso-arrows of the category $\QP$.

% A proof of the following can be obtained by translating the proof
%of \cite[Theorem 3.4]{MS201X} from the language of MV-algebras into
%the language of unital $\ell$-groups
% via the equivalence proven in \cite{Mu1986} (see also \cite{MS201Xa}).

 \medskip
  A proof of the following  result can be obtained from
\cite[\S 3]{Mu2011}:

\begin{theorem}[\bf Duality]\label{Theo_Baker-Beynon}
Let the functor $\McN\colon \QP\rightarrow \FP$ be defined by
   $$
     \begin{tabular}{ll}
       Objects:& For  $P\in \QP$  a polyhedron,\\& $\McN(P)$ is the set of
       all     \,\Zed-maps
       from   $P$ into $\R$.\\[0.2cm]
       Arrows:& For $\eta\colon P\rightarrow Q$  a \Zed-map,\\&
$\McN(\eta)(f)=f\circ \eta$,
        for each $f\in \McN(Q)$.
     \end{tabular}
   $$
Then $\McN$ yields a duality between the
categories $\QP$ and $\FP$.
Stated otherwise,  $\McN$ is a categorical
equivalence between $\QP$ and the opposite category of $\FP$.
\end{theorem}

\begin{lemma}\label{Lem_ImagesZmorph}
Suppose   $P\subseteq \R^{n}$ and $Q\subseteq \R^{m}$
  are  polyhedra and $\eta\colon P\rightarrow Q$ is a \Zed-map.
We then have $$\McN(Q)/{\rm ker}(\McN(\eta))\cong\McN(\eta(P)).$$
\end{lemma}
\begin{proof}
  % (i) Let $\Delta$  be a triangulation of $P$ such that $\eta$ is
%linear on each simplex of  $\Delta$.
%   %
%   Let $T$ be a simplex of $\Delta$, and let $\{v_1,\ldots,v_r\}$ be
%the vertices of $T$.
%   %
%   Then
%    $\eta(T)=\conv(\eta(v_1),\ldots,\eta(v_r))$,
%   and $\eta(T)$ is the union of simplexes whose vertices are
%contained in $\{\eta(v_1),\ldots,\eta(v_r)\}$.
%   %
%   Since all points
%   $\{v_1,\ldots,v_r\}$ are rational and $\eta$ is linear on $T$ with
%integer coefficients,
%   then
%   all points  $\{\eta(v_1),\ldots,\eta(v_r)\}$ are rational.
%   %
%This shows that $\eta(T)$ is a rational polyhedron.
%   %
%The desired conclusion now follows
%  from the fact that $\eta(P)=\bigcup\{\eta(T)\mid T\in\Delta\}$.
%   \medskip
%
 Observe that $f\in {\rm ker}(\McN(\eta))$ iff
$f\circ \eta (P)=f(\eta(P))=\{0\}$.
   Therefore the kernel of the onto map $h\colon\McN(Q)\rightarrow
\McN(\eta(P))$
given  by $f\mapsto f\upharpoonright \eta(P)$,
   coincides with ${\rm  ker}(\McN(\eta))$.
In conclusion,
$\McN(Q)/{\rm ker}(\McN(\eta))\cong\McN(\eta(P))$
\end{proof}

%% \begin{corollary}\label{Cor_Subalgebras}
%%   Let $(G_1,u_1)$ and $(G_2,u_2)$  be finitely presented unital
%$\ell$-groups,
%%   and $h\colon G_1\rightarrow G_2$  a homomorphism.
%%   %
%%   Then $(h(G_1),u_2)$ is a finitely presented unital $\ell$-group.
%% \end{corollary}
%%
%% \begin{proof}
%%   By Theorem \ref{Theo_Baker-Beynon}, we can assume
%$(G_1,u_1)=\McN(P_1)$ and $(G_2,u_2)=\McN(P_2)$,
%%   for some  polyhedra $P_1\subseteq\R^{n}$ and $P_2\subseteq \R^{m}$.
%%
%%   Let  $\eta\colon P_2\rightarrow P_1$ be the unique \Zed-map such
%that $\McN(\eta)=h$,
%%   as  given by Theorem \ref{Theo_Baker-Beynon}.
%%   %
%%   By Lemma \ref{Lem_ImagesZmorph}(ii),
%%   $$(h(G_1),u_2)\cong G_1/{\rm ker}(h)\cong\McN(P_1)/{\rm
%ker}(\McN(\eta))\cong\McN(\eta(P_2)).$$
%%    %
%%
%%    One more application of Theorem \ref{Theo_Baker-Beynon}
%completes the proof.
%%
%% \end{proof}

%
%%%%%%%%%%%%%%%%%%%%%%%%%%
\subsection{Combinatorics of rational polyhedra}
%%%%%%%%%%%%%%%%%%%%%%%%%%%%

  Building on  \cite{CM20XX} and \cite{MuXXXX},
in this section we introduce a functor
from  the category of abstract simplicial complexes   with weighted
 vertices into the category of rational polyhedra.
This will be used
to  construct
limits and co-limits of rational polyhedra. Using
the dual equivalence of Theorem \ref{Theo_Baker-Beynon}  we will then characterize
finitely generated subalgebras of free unital $\ell$-groups.

\smallskip

%%%%%%%%%%%%%%%%% definition of abstract simplicial complex
Let us recall that a {\it (finite) abstract simplicial complex} is a
pair $({V},\Sigma)$,
where ${V}$ is a finite
  set, whose elements are called the {\it vertices} of $({V},\Sigma)$,
and $\Sigma$ is a collection of subsets of ${\mathscr V}$ whose union
is ${V}$,
having the property that every subset of an element of $\Sigma$ is
again an element of $\Sigma$.
%%%%%%%%%%%%%%%%%%%%%%%%%%%

%%%%%%%%%%%%%% definition of weight abstract simplicial complex
A {\it weighted abstract simplicial complex} is a triple $({
V},\Sigma, \omega)$
where $({ V},\Sigma)$ is an abstract simplicial complex
and $\omega$ is a map of ${ V}$ into the set $ \{1,2,3,\ldots\}.$
%%%%%%%%%%%%%%%%%%%%%%%%%

%%%%%%%%%%%%%%%%%% Definition of d-morphism
Given two weighted abstract simplicial complexes $ \mathfrak{W}=
( V,\Sigma,\omega)$ and $\mathfrak{W}' = (
V',\Sigma',\omega') $
a simplicial map $\gamma\colon  V \rightarrow  V',$
(that is, $\gamma(S)\in \Sigma'$ for each $S\in\Sigma$)
  is a morphism from $\mathfrak{W}$  into $\mathfrak{W}'$,
if  $\omega'(\gamma(v))$ divides $\omega(v)$  for all  $v\in { V}$.
%%%%%%%%%%%%%%%%%%%%%%%

%%%%%%%%%%%%%%%%% Definition of the category WASC
We denote $\wasc$ the category of weighted abstract simplicial complexes.
%%%%%%%%%%%%%%%%%%%%%%%%%%%

%%%%%%%%%%%%%%%%% Definition of combinatorially isomorphic
It is easy to see that two weighted abstract simplicial complexes
$ \mathfrak{W}= ( V,\Sigma,\omega)$ and $\mathfrak{W}' =
( V',\Sigma',\omega')$
are isomorphic in $\wasc$ iff
there is   a one-one map  $\gamma$
  from $ V$ onto ${ V}'$
having the following properties:
\begin{itemize}
\item $\omega'(\gamma(v))=\omega(v)$  for all  $v\in { V}$,
and
\item $\{w_{1},\ldots,w_{k}\}\in \Sigma$ iff
$\{\gamma(w_{1}),\ldots,\gamma(w_{k})\}\in \Sigma'$
for each   $\{w_{1},\ldots,w_{k}\}\subseteq  V$.
\end{itemize}
When this is the case
  we   say that $\mathfrak{W}$ and $\mathfrak{W}'$ are {\it
combinatorially isomorphic}.
%%%%%%%%%%%%%%%%%%%%%%%%%%%%%%

\medskip
%%%%%%%%%%%%%%%%%%%%%%%%% Definition of canonical realization
Let $\mathfrak{W}=({ V},\Sigma, \omega)$ be a weighted
abstract simplicial complex with vertex set  ${
V}=\{v_{1},\ldots,v_{n}\}$.
Let
  $e_{1},\ldots,e_{n}$ be the standard basis vectors of
  $\mathbb R^{n}$.
We then use the notation $\Delta_{\mathfrak{W}}$
for the complex whose vertices are the following points
 in $\mathbb R^n$
$$
   v'_{1} = e_{1}/\omega(v_{1}),\ldots,v'_{n}=e_{n}/\omega(v_{n}),
$$
and whose $k$-simplexes ($k=0,\ldots,n$) are given by
$$
   \conv(v'_{i(0)},\ldots, v'_{i(k)})\in \Delta_{\mathfrak{W}}\quad
\text{ iff }\quad
   \{ v_{i(0)},\ldots, v_{i(k)}\}\in \Sigma.
$$
Trivially, $\Delta_{\mathfrak{W}}$ is a regular triangulation of the
polyhedron $|\Delta_{\mathfrak{W}}|\subseteq [0,1]^{n}$.  The
polyhedron $|\Delta_{\mathfrak{W}}|$ is called the {\it geometric
realization} of $\mathfrak{W}$ and will be denoted
$\Pol(\mathfrak{W})$.

  %%%%%%%%%%%%%%%%%%%%%%%%%%%%%%%%%%%%%

\medskip

In order to extend $\Pol$ to a functor from $\wasc$ into $\QP$
we prepare

\begin{lemma}\cite[Lemma 3.7]{Mu2011}\label{Lem-LinearMap}
   Let  $S=\conv(x_{1},\ldots,x_{k})\subseteq \R^{n}$ be a regular
${(k-1)}$-simplex, and
   $\{y_{1},\ldots,y_{k}\}$ a set of rational points in $\R^{n}.$
   Then the following conditions are equivalent:
   \begin{itemize}
     \item[(i)]  For each $i=1,\ldots,k,$ $\den(y_{i})$  is a divisor
of $\den(x_{i})$. \smallskip
     \item[(ii)] For some integer matrix  $M\in\Zed^{n\times n}$ and
integer vector  $b\in\Zed^{n}$,  $M x_{i}+b=y_{i}.$
   \end{itemize}
\end{lemma}

\begin{corollary}\label{Cor_Div_Denominators}
   Let $P\subseteq \R^{n}$ and $Q\subseteq \R^{m}$ be polyhedra
   and $\eta \colon  P\rightarrow Q$  a $\Zed$-map.
   Then for every rational point $x\in P,\,\,\,\,$
   $
     \den(\eta (x))\mbox{ divides }\den(x).
   $
\end{corollary}

\begin{corollary}\label{Cor_ExtensionToZed}
   Let $P\subseteq\R^{n}$ be a  polyhedron, $\Delta$  a regular
triangulation of $P$ and
   $f\colon
   \ver(\Delta)\rightarrow \mathbb{Q}^{m}$ a map
    such that $\den(f(v))$ divides $\den(v)$ for each $v\in\ver(\Delta)$.
   Then there exists a unique $\Zed$-map $\eta\colon P\rightarrow
\R^{m}$ satisfying the following two conditions:
   \begin{enumerate}
     \item $\eta$ is linear on each simplex of $\Delta$;
     \item $\eta\restrict{\ver(\Delta)}=f$.
   \end{enumerate}
\end{corollary}

\bigskip

%%%%%%%%%%%%%%%%%% Definition of the arrow part of the functor P
\noindent Let $\mathfrak{W}= ( V,\Sigma,\omega)$ and $\mathfrak{W}' =
( V',\Sigma',\omega')$ be  weighted abstract simplicial
complexes, with
${ V}=\{v_{1},\ldots,v_{n}\}$ and  ${
V'}=\{v'_{1},\ldots,v'_{m}\}$.
Let  $\gamma\colon  V \rightarrow  V'$ be a morphism
of weighted abstract simplicial complexes.
   Corollary \ref{Cor_ExtensionToZed} yields
    a unique \Zed-map $\Pol(h)\colon
  \Pol(W)\rightarrow \Pol(W')$  with the following properties:
  \begin{itemize}
\item
$\Pol(h)$  is  linear on each simplex $S$ of $\Delta_{W}$,  and
\item
for each $i=1,\ldots,n,\,\,\,$
$\eta(e_{i}/\omega(v_{i}))=e_{j}/\omega'(v'_{j})$ whenever
$\gamma(v_i)=v'_{j}$.
\end{itemize}
%%%%%%%%%%%%%%%%%%%%%%%%%%%

%%%%%%%%%%%%%%%%% Definition of the funtor P
As a consequence,   $\Pol$ is a faithful functor from $\wasc$ into $\QP$.
%%%%%%%%%%%%%%%%%%%%%%%%%%%%%%
\smallskip

%%%%%%%%%%%%%%%%%%%%%% Definition of Skeleton
For every regular complex $\Delta$,  the {\it skeleton} of  $\Delta$
is the  weighted abstract simplicial complex
$\mathfrak{W}(\Delta)=(V,\Sigma,\omega)$ given by the following stipulations:
\begin{enumerate}
       \item  $V=\ver(\Delta)$.
       \item   For every  $v\in\ver(\Delta)$,
$\omega(v)=\den(v).$
\item For every subset $W=\{w_{1},\ldots,w_{k}\}$
of $V$, $W\in \Sigma$ iff $\conv(w_{1},\ldots,w_{k})\in\Delta.$
\end{enumerate}
%%%%%%%%%%%%%%%%%%%%%%%%%%%%%%

Let  $\{v_1,\ldots,v_m\}$ be the vertices of
a regular triangulation  $\Delta$  of a polyhedron $P$.
Let
\begin{equation}\label{Eq:IsoinCube}
\iota_{\Delta}\colon P\rightarrow \Pol(\mathfrak{W}(\Delta))
\end{equation}
 be
the unique \Zed-map given by Corollary \ref{Cor_ExtensionToZed} which
is linear on each simplex of $\Delta$ and also satisfies
$\iota_\Delta(v_i)=e_i/\den(v_i)$.  Then
$\iota_\Delta$ is a \Zed-homeomorphism.
Since $\Pol(\mathfrak{W}(\Delta))\subseteq [0,1]^m$, as a byproduct we obtain that each rational polyhedron is \Zed-homeomorphic to a polyhedron contained in some $m$-cube.
%%%%%%%%%%%%%%%%%%%%%%%%%%%%%%%%%%%

%%%%%%%%%%%%%%%% Properties of Pol
We have just proved
  that for each object $P$ of $\QP$ there exists an object
$\mathfrak{W}$ of $\wasc$ such that $\Pol(\mathfrak{W})$ is
isomorphic to $P$ in $\QP$.
Yet, $\Pol$ does not define an equivalence between the
categories $\wasc$ and $\QP$,
because $\Pol$  is not full:

\begin{example}
   Let the \Zed-map
   $\eta\colon[1/4,1/3]\rightarrow[0,2/3]$ be defined by $\eta(x)=8x-2$.
   Let the rational point $a\in [1/4,1/3]$ be
   such that $[1/4,a]$ is a regular $1$-simplex.
  Writing $a=k/l$ for $k, l\in \{1,2,\ldots\}$ with  $\gcd(k,l)=1$, it
  follows that   $l=4k-1$,  whence $\eta(a)=8(k/(4k-1))-2=2/(4k-1)$.
   Thus $[0,2/(4k-1)]=\eta([1/4,a])$ is not regular.
As a consequence,  there is no
    regular $\Delta$ triangulation of $[1/4,1/3]$ such that $\eta(\Delta)$
    is a regular triangulation of the simplex  $[0,2/3]$.
   Since for each morphism $\gamma\colon \mathfrak{W} \rightarrow
\mathfrak{W}'$ of weighted abstract simplicial complexes,
   $\Pol(h)$ satisfies
$\Pol(h)(\Delta_{\mathfrak{W}})\subseteq\Delta_{\mathfrak{W}'}$,
   we conclude that $\Pol$ is not full.
\end{example}

%%%%%%%%%%%%%%%%%%%%%%%%%%%%%%%%%%%%%%%%%%%%%%%%%%

%%%%%%%%%%%%%%%%%%%%%%%%%%%%%%%%%%%%%%%%%%%%%%%%%%%%%%%%%%%%%%%%%%%%%%%%%%%%%%%%%
%%%%%%%%%%%%%%%%%%%%%%%%%%%%%%%%%%%%%%%%%%%%%%%%%%%%%%%%%%%%%%%%%%%%%%%%%%%%%%%%%
\section{Properties of the category of Rational Polyhedra}

In this section we will study some properties of the category $\QP$
of rational polyhedra,
  and the dual properties of the category of finitely presented unital
  $\ell$-groups.
%
%To help the reader,  we give complete proofs of all
%these properties.

%%%%%%%%%%%%%%%%%%%%%%%%%%%%%%%%%%
\subsection{$\Zed$-maps and limits}
%%%%%%%%%%%%%%%%%%%%%%%%%%%%%%%%%%
The category of unital $\ell$-groups is small complete and small co-complete.
This is a consequence of the categorical equivalence between unital
$\ell$-groups and the equational class of MV-algebras,
  \cite[Theorem 3.9]{Mu1986}.\footnote{The small completeness (small co-completeness) for equational classes of algebras follows from Birkhoff Theorem (see \cite[Theorem 11.9]{BS1981}) and the construction of limits (co-limits) by products and equalizers (co-products and co-equalizers) (see \cite[\S 5.2. Theorem 1]{McL1969}).}

It follows that finitely presented unital $\ell$-groups are closed
under finite co-limits,
whence, by Theorem \ref{Theo_Baker-Beynon},    $\QP$ is closed under
finite limits.

We next
construct  finite limits in  $\QP$.
This  will be the key tool
  to describe monic and epic arrows in the category $\QP$
   in Theorem \ref{Theo_monicepi}.

\begin{theorem}[\bf Limits]\label{Teo_QP_SmallComplete}
   The category $\QP$ is closed under finite limits.
\end{theorem}
\begin{proof} In view of \cite[\S V.2. Corollary 2]{McL1969}, we
only need to prove that $\QP$ has finite products and equalizers.
   \medskip

   \noindent {\it Finite products}: It is easy to see that the set $\{1\}$ is the terminal object $\QP$.
     Therefore, $\QP$ admits the empty product.

     Suppose that $P\subseteq \R^{n}$ and $Q\subseteq \R^{m}$ are polyhedra.
     The   product of $P$ and $Q$ in $\Set$,
     $$
       P\times Q=\left\{ (x,y)\in\R^n\times\R^m\mid x\in P \text{ and } y\in Q \right\}
     $$
     is   a rational polyhedron.
     Using Lemma \ref{Lem_CarZed}, it is easy to see that the
projections $\pi_{P}$ and $\pi_{Q}$ are $\Zed$-maps.

     Suppose that $R\subseteq \R^{l}$ is a polyhedron and $\eta\colon
R\rightarrow P$ and $\mu\colon R\rightarrow Q$ are \Zed-maps.

     An application of Lemma \ref{Lem_CarZed}  shows  that
     the unique map $(\eta,\mu)\colon R\rightarrow P\times Q$ such that
     $\pi_P\circ (\eta,\mu)=\eta$ and $\pi_Q\circ (\eta,\mu)=\mu$ is a \Zed-map.
Thus  the cone $\xymatrix{P&P\times
Q\ar[l]_{\pi_{P}}\ar[r]^{\pi_{Q}}& Q}$ is universal in $\QP$,
and   $P\times Q$ is the product of $P$ and $Q$ in the category $\QP$.
     \medskip

   \noindent {\it Equalizers}: Let $P\subseteq \R^{n}$ and $Q\subseteq
\R^{m}$ be polyhedra and $\eta,\mu :P\rightarrow Q$   two
     $\Zed$-maps. Let
     $$
       E=\left\{ x\in P\mid\mu(x)=\eta(x)\right\}
     $$
     be the equalizer of $\mu$ and $\eta$ in $\Set$.
     To see that $E$ is a polyhedron, for each $i=1,\ldots,m,$
      let $f_i\colon P\rightarrow \R$ be defined by
      $f_i=|\xi_i\circ \mu-\xi_i\circ\eta|$, where $\xi_i\colon
Q\rightarrow \R$ are
      the coordinate maps.
     Writing $f=f_1+\cdots+f_m$ it follows that $E=f^{-1}(0)$.
By Lemma \ref{Lem_ZeroP},   $E$ is a rational polyhedron.

     If $R\subseteq \R^{l}$ is a polyhedron and $\nu\colon
R\rightarrow P$ is a \Zed-map such that $\nu\circ \eta=\nu\circ\mu$
then $\nu(R)\subseteq E$.
     Since inclusions are  $\Zed$-maps, the proof is complete.
\end{proof}

%As we have seen, finite limits in $\QP$ are special instances
%of limits in $\Set$. For instance,
%  pullbacks of (always rational)
%   polyhedra are just particular instances of the fiber product in $\Set$.

\begin{theorem}\label{Theo_monicepi}
   Let $P\subseteq \R^{n}$ and $Q\subseteq \R^{m}$ be polyhedra and
$\eta : P\rightarrow Q$ a $\Zed$-map.
   Then
   \begin{itemize}
     \item[(i)] $\eta$ is monic in $\QP$ iff it is one-one.
     \item[(ii)] $\eta$ is epic in $\QP$ iff it is onto.
   \end{itemize}
\end{theorem}
\begin{proof}
   (i) For the nontrivial direction, suppose
   $\eta$ is a monic arrow in $\QP$.
   If $\eta$ is not one-one (absurdum hypothesis) let
   $$
     E=\left\{(x,y)\in P\times P\mid\eta(x)=\eta(y)\right\}.
   $$
The proof of Theorem \ref{Teo_QP_SmallComplete} shows that
   $E$ is the equalizer of $\eta\circ\pi_{1}$ and $\eta\circ\pi_{2}$, where
   $\pi_{1}$ and $\pi_{2}$ are the projections of $P\times P$ onto $P$.
   By our assumption, there exists a   rational point $(x,y)\in E$
such that $x,y\in P$ and $x\neq y$.
   Let $d=\den(x)\cdot \den(y)$ and $S=\left\{ \frac{1}{d}
\right\}\subseteq \R$.
   By Lemma \ref{Lem-LinearMap}, the maps $\mu_{1},\mu_{2}:S\rightarrow P$
   defined by $\mu_{1}(\frac{1}{d})=x$ and $\mu_{2}(\frac{1}{d})=y$
are \Zed-maps.
   Clearly, $\eta\circ\mu_{1}=\eta\circ\mu_{2}$ and $\mu_{1}\neq\mu_{2}$,
   a contradiction with the assumption that $\eta$ is monic in $\QP$.

   \medskip

   (ii) For the nontrivial direction, suppose that $\eta$ is epic in
$\QP$ but is not onto $Q$ (absurdum hypothesis).
   By Lemma \ref{Lem_ImagesZmorph},  $\eta(P)\subseteq Q$ is a polyhedron.
   By Lemma \ref{Lem_Triang-Subset},
   there exists  a regular triangulation
   $\Delta_{Q}$
   of $Q$ such that $\{S\in\Delta_{Q}\mid \ver(S)\subseteq\eta(P)\}$
is a regular triangulation of $\eta(P)$.

Let  $\mu_{1}\colon Q\rightarrow Q\times [0,1]$ be defined by
   $$
     \mu_{1}(v)=(v,0).
   $$
Let  $\mu_{2}$ be   the unique $\Zed$-map of Corollary
\ref{Cor_ExtensionToZed},
satisfying the following conditions
  for each $S\in\Delta_{Q}$ and every $v\in \ver(S)$:
   $$
     \mu_{2}(v)= \left\{\begin{tabular}{cl}
       $(v,0)$  & if $v\in \eta(P),$ \\
       $(v,1)$ & if  $v\notin \eta(P),$
     \end{tabular}\right.
   $$
  with $\mu_{2}$  being
   linear on each simplex of $\Delta_Q$.

   It is easy to see that $\mu_{1}\circ\eta=\mu_{2}\circ\eta$ but
$\mu_{1}\neq\mu_{2}$,
   a contradiction with the assumption
    that $\eta$ is epic. \end{proof}

The foregoing  result
is well known to the specialist:  in particular,
 (ii) can also be derived from
Theorem \ref{Theo_Baker-Beynon}
 in combination with \cite[Lemma 3.8]{Mu2011}.
   We have given a proof for the
sake of completeness.

\subsection{Cantor-Bernstein-Schr\"oder theorem}\label{SecZedHomeo}

In the previous subsection we
have characterized monic and epic arrows  in the category $\QP$.
In \cite[Proposition 3.15]{Mu2011} a characterization of iso-arrows
  in $\QP$ is given in terms of preservation of denominators of rational points.
Using this result,  in Corollary \ref{Cor_CBS}
we  will prove a (dual) Cantor-Bernstein-Schr\"oder theorem for finitely
presented unital $\ell$-groups.

By Theorem \ref{Theo_Baker-Beynon} we immediately have

\begin{lemma}\label{Lem_ZedHomeo}
   Let $P\subseteq \R^{n}$ and $Q\subseteq \R^{m}$ be polyhedra and
$\eta \colon P\rightarrow Q$ a  \Zed-map.
   Then $\eta$ is a \Zed-homeomorphism in $\QP$ iff
$\McN(\eta)\colon\McN(Q)\rightarrow \McN(P)$ is an isomorphism in
$\FP$.
\end{lemma}

\begin{theorem}\cite[Proposition 3.5]{Mu2011}\label{Theo_TriangZedHomeo}
   Let $P\subseteq \R^{n}$ and $Q\subseteq \R^{m}$ be polyhedra and
   $\eta \colon  P\rightarrow Q$  a one-one \Zed-map of $P$ onto $Q$.
   Then the following conditions are equivalent:
\begin{itemize}
   \item[(i)] $\eta$ is a \Zed-homeomorphism.
   \item[(ii)] $\den(\eta(x))=\den(x)$ for each rational point $x\in P$.
   \item[(iii)] For each regular simplex $S\subseteq P$, $\eta(S)$ is
a regular simplex of $Q$ and $\den(\eta(x))=\den(x)$ for each  $x\in
\ver(S)$.
   \item[(iv)] For some (equivalently,  for every) regular
triangulation $\Delta$ of $P$ such that $\eta$ is linear on each
simplex of $\Delta$, $\eta(\Delta)$ is a regular triangulation of $Q$
and $\den(\eta(x))=\den(x)$ for each  $x\in \ver(\Delta)$.
\end{itemize}
\end{theorem}

\begin{definition}
   Let $P\subseteq \R^{n}$ and $Q\subseteq \R^{m}$ be polyhedra and
    $\eta \colon P\rightarrow Q$ a $\Zed$-map.
   Then $\eta$ is a {\it strict $\Zed$-map} if it is a
$\Zed$-homeomorphism onto its range.
\end{definition}

 From Theorem~\ref{Theo_TriangZedHomeo} we obtain:

\begin{corollary}\label{Cor-PreserDen}
   Let $P\subseteq \R^{n}$ and $Q\subseteq\R^{m}$ be
    polyhedra and $\eta\colon P\rightarrow Q$  a one-one  $\Zed$-map.
   Then the following conditions are equivalent:
   \begin{itemize}
     \item[(i)] $\eta$ is a strict \Zed-map.
     \item[(ii)] $\den(\eta(x))=\den(x)$ for each rational point $x\in P$.
     \item[(iii)] For each regular simplex $S\subseteq P$, $\eta(S)$
is a regular simplex of $Q$ and $\den(\eta(x))=\den(x)$ for each
$x\in \ver(S)$.
     \item[(iv)] For some (equivalently, for every) regular
triangulation $\Delta$ of $P$ such that $\eta$ is linear
     on each simplex of $\Delta$, $\eta(\Delta)$ is a regular
triangulation of $\eta(P)$ and $\den(\eta(x))=\den(x)$ for each
$x\in \ver(\Delta)$.
   \end{itemize}
\end{corollary}

By Theorem \ref{Theo_monicepi}, every strict $\Zed$-map is a monic
$\Zed$-map, but the converse does not hold in general.
From Theorem \ref{Theo_Baker-Beynon}, monic \Zed-maps correspond to epi unital $\ell$-homomorphisms.
The following theorem
shows that strict \Zed-maps correspond to onto (or equivalently regular epi)
 unital
$\ell$-homomorphisms:

\begin{theorem}\label{Theo_StricOnto}
   Let $P\subseteq \R^{n}$ and $Q\subseteq \R^{m}$ be polyhedra and
$\eta\colon P\rightarrow Q$ be a \Zed-map.
   Then $\eta$ is a strict $\Zed$-map iff
    $\McN(\eta)\colon\McN(Q)\rightarrow \McN(P)$ is an onto map.
\end{theorem}

\begin{proof}
$(\Rightarrow)$
   In order to prove that $\McN(\eta)$ is onto $\McN(P)$, let $f\in\McN(P)$.
   By Proposition \ref{proposition:poly} and Lemma
\ref{Lem_StableTriangul}, there exists  a regular triangulation
$\Delta_P$ of $P$ such that $\eta$ and $f$ are linear over each
simplex $S\in\Delta.$
   Since $\eta$ is a strict $\Zed$-map, by Theorem \ref{Theo_TriangZedHomeo},
    $\eta(\Delta_P)$ is a regular triangulation of $\eta(P)$.
   By  Lemma \ref{Lem_Triang-Subset} there exists    a regular triangulation
   $\Delta_Q$
   of $Q$ such that the set  $\nabla=\{S\in\Delta_Q\mid \ver(S)\subseteq
\eta(P)\}$ is a subdivision of $\eta(\Delta_P)$. Then
$\eta^{-1}\colon\eta(P)\rightarrow P$ is linear over each simplex of
   $\nabla$. By Theorem \ref{Theo_TriangZedHomeo},
    $\eta^{-1}(\nabla)$ is a regular triangulation of $P$ and  is a
subdivision of $\Delta_P$.

Let  $g\in\McN(Q)$
be
  uniquely determined by the following conditions:
   \begin{enumerate}
     \item  $g$ is linear over each simplex $S\in\Delta_Q,$
     \item  for every $v\in \ver(\Delta_Q)$,
       $$
         g(v)=\left\{\begin{tabular}{ll}
           $f(\eta^{-1}(v))$ & if $v\in\eta(P),$ \\
           $0$ & if  $v\not\in\eta(P).$
         \end{tabular}\right.
       $$
   \end{enumerate}
   The existence and uniqueness of   $g$
is ensured by  Lemma~\ref{Lem-LinearMap}.
It follows that
  $g\circ\eta$ is linear over each $S\in\eta^{-1}(\nabla)$.
   For every $x\in \eta^{-1}(\nabla)$ we can write
   $$
      g\circ\eta(x)=g(\eta(x))=f(\eta^{-1}(\eta(x)))=f(x).
   $$
   Therefore, $g\circ\eta=h(g)=f$,  and  $\McN(\eta)$ is onto.
   \medskip

$(\Leftarrow)$
Assume  $\McN(\eta)$ is an onto map.
   By Lemma \ref{Lem_ImagesZmorph}(ii),
   $$
   \McN(P)\cong\McN(Q)/\ker(\McN(\eta))\cong\McN(\eta(P)).
   $$
By Lemma~\ref{Lem_ZedHomeo}, $\eta\colon P\rightarrow\eta(P)$ is a
$\Zed$-homeomorphism.
\end{proof}

\begin{theorem}
   Let $P\subseteq \R^{n}$ be a polyhedron and $\eta\colon P\rightarrow P$
   a one-one (equivalently,  a monic) $\Zed$-map.
   Then $\eta$ is a $\Zed$-homeomorphism.
\end{theorem}

\begin{proof}
   By Theorem \ref{Theo_TriangZedHomeo}, it is enough to prove that
$\eta$ preserves denominators and   is onto $P$.
   For each $k=1,2,\ldots$, let $P_{k}=\{x\in P\cap \mathbb{Q}^{n}\mid
\den(x)=k\}$.
   Since $P$ is a bounded set, each $P_{k}$ is a finite set.

\smallskip
\noindent {\it Claim}: $\eta$ preserves denominators.
Equivalently,  $\eta(P_{k})=P_{k}$ for each $k=1,2,\ldots$.

   %%%%%%%%%%%%%%%%%%%% Proof by induction
The proof is by induction on $k$.
   %
   %%%%%%%%%%%%%%%%%%%% Initial case
   For the basis  case, let $x\in P_{1}$.
   Then $\den(\eta(x))$ divides $\den(x)=1$, i.e.  $\eta(P_{1})\subseteq P_{1}$.
   Since $\eta$ is one-one and $P_{1}$ is finite, $\eta(P_{1})=P_{1}.$

   %%%%%%%%%%%%%%%%%%%% Induction step
   For the induction step, suppose that for every $j<k$, $\eta(P_{j})=P_{j}$.
   Let $x\in P_{k}$. %
   Since $\den(\eta(x))$ divides $k$, then $\den(\eta(x))\leq k$.
  Assume  $\eta(x)\not\in P_{k}$ (absurdum hypothesis).
   Then $\eta(x)\in P_{k'}$ for some $k'<k$.
   By hypothesis, there exists $y\in P_{k'}$ such that $\eta(y)=\eta(x)$.
   Since $\den(x)=k'\neq k= \den(y)$,  then
    $x\neq y$, thus contradicting the fact that $\eta$.
    Our claim is settled.
   %%%%%%%%%%%%%%%%%%% End of the proof by induction

  \smallskip

   To see that $\eta$ is onto $P$, first let us
   observe that since $P$ is a rational polyhedron,
   $$
   P={\rm cl}(\bigcup_{k\geq 1} P_{k}),
   $$
   where ${\rm cl}$ denotes topological closure.
The continuity of $\eta$ now yields
   \begin{eqnarray}
   \nonumber  \eta(P) &=& \textstyle\eta({\rm cl}(\bigcup_{k\geq 1}
P_{k})) = {\rm cl}(\eta(\bigcup_{k\geq 1}P_{k}))\\
   \nonumber          &=& \textstyle{\rm cl}(\bigcup_{k\geq
1}\eta(P_{k}))  = {\rm cl}(\bigcup_{k\geq 1} P_{k})\\
   \nonumber          &=& P,
   \end{eqnarray}
   i.e.  $\eta$ is onto $P$. The proof is complete.
\end{proof}

\begin{theorem}
   Let $P\subseteq \R^{n}$ and $Q\subseteq \R^{m}$ be polyhedra,
   and $\eta
\colon P\rightarrow Q$ and $\mu\colon Q\rightarrow P$ be one-one
\Zed-maps.
   Then $P$ is $\Zed$-homeomorphic to $Q$.
\end{theorem}
\begin{proof}
   Since $\nu=\mu\circ\eta$ is a one-one $\Zed$-map from $P$ into
itself, by the previous theorem $\nu$ is a \Zed-homeomorphism.

We claim that $\nu^{-1}\circ\mu\colon Q\rightarrow P$ is the inverse of $\eta$.
   Denoting by $Id_{P}$ and $Id_{Q}$
    the identity maps over $P$ and $Q$, we get
   $$
(\nu^{-1}\circ\mu)\circ\eta=\nu^{-1}\circ(\mu\circ\eta)=\nu^{-1}\circ(\nu)=Id_{P}.
   $$
  Symmetrically,
   $
     \mu\circ(\eta\circ\nu^{-1}\circ\mu)=\mu=\mu\circ Id_{Q}.
   $
   Since $\mu$ is a one-one $\Zed$-map, then
$\eta\circ(\nu^{-1}\circ\mu)=Id_{Q}.$
\end{proof}

   As a corollary we obtain a (dual) Cantor-Bernstein-Schr\"oder
theorem for finitely presented unital $\ell$-groups:
\begin{corollary}\label{Cor_CBS} For any finitely presented unital $\ell$-groups
   $(G_1,u_1)$ and $(G_2,u_2)$ the following conditions  are equivalent:
   \begin{enumerate}
     \item  $(G_1,u_1)$ and $(G_2,u_2)$ are isomorphic.
     \item  There are onto homomorphisms
$f\colon(G_1,u_1)\rightarrow(G_2,u_2)$ and
$g\colon(G_2,u_2)\rightarrow(G_1,u_1)$.
     \item  There are epic homomorphisms
$h\colon(G_1,u_1)\rightarrow(G_2,u_2)$ and
$l\colon(G_2,u_2)\rightarrow(G_1,u_1)$.
   \end{enumerate}
\end{corollary}

It is easy to check
  that if in item (ii) we replace ``onto'' by ``one-one'' the result
is no longer valid.

%%%%%%%%%%%%%%%%%%%%%%%%%%%%%%%%%%%%%%%%%%%%%%%%%%%%%%%%%%%%%%%%%%%%%%%%%%
\subsection{Co-limits}
%%%%%%%%%%%%%%
Having constructed finite limits in the category of rational polyhedra,
we devote  this section to a special type of co-limits that will find
use in the rest of the paper.

{}From
  Lemmas
\ref{Lem_Triang-Subset},
\ref{Lem_Sub-Triang-Subset},
\ref{Lem_StableTriangul}
and Theorem \ref{Theo_TriangZedHomeo} we obtain

\begin{lemma}\label{Lem_TrianofVstrictmap}
   Let $P\subseteq \R^{n}$, $Q\subseteq \R^{m}$ and $D\subseteq
\R^{k}$ be polyhedra and
   $\eta\colon P\rightarrow Q$ and $\mu\colon P\rightarrow D$ be
strict $\Zed$-maps. Then there exist
   regular triangulations  $\Delta_P$, $\Delta_Q$, $\Delta_R$  of $P$, $Q$,
    $R$  satisfying the following conditions:
   \begin{itemize}
    \item[(i)] $\eta$ and $\mu$ are linear over each simplex of $\Delta_{P}$;
    \item[(ii)] $\eta(\Delta_P)$ is a full subcomplex of $\Delta_Q$;
    \item[(iii)] $\mu(\Delta_P)$ is a full subcomplex of $\Delta_R$.
   \end{itemize}
   Moreover, for any regular triangulations
   $\nabla_P$, $\nabla_Q$, $\nabla_R$
     of $P$, $Q$, $R$  we may insist that
      $\Delta_P$, $\Delta_Q$,  $\Delta_R$ subdivide
    $\nabla_P$, $\nabla_Q$,  $\nabla_R$, respectively.
\end{lemma}

\begin{theorem}[\bf Co-limits]\label{Theo-Colimits}
     $\QP$ has finite co-products and pushouts of strict \Zed-maps.
\end{theorem}
\begin{proof}
{\it Finite co-products}:
  Considered as a rational polyhedron,
  the empty set  is the intial object of $\QP$, and is also  the empty
co-product.
Next, for any polyhedra
  $P\subseteq \R^{m}$ and $Q\subseteq \R^{n}$  let
  $k={\rm max}\{m,n\}$ and
   $$
     C=\left\{ (x,0_{m+1},\ldots,0_{k+1})\mid x\in P\right\}\cup
\left\{(y,1_{n+1},\ldots,1_{k+1})\mid y\in Q\right\}.
   $$
   Then
   $C$ is the co-product of $P$ and $Q$ in $\Set$ and the inclusion maps
   $$
    \begin{tabular}{ccc}
      $i_{P}\colon P\rightarrow C$ & and & $i_{Q}\colon Q\rightarrow C$ \\
      $i_{P}(x)=(x,0_{m+1},\ldots,0_{k+1})$ & &
      $i_{Q}(y)=(y,1_{n+1},\ldots,1_{k+1})$
    \end{tabular}
    $$
    are $\Zed$-maps.

   Let $R\subseteq \R^{l}$ be a polyhedron and $\eta\colon
P\rightarrow R$ and $\mu\colon Q\rightarrow R$ be \Zed-maps.
By
  Lemma \ref{Lem_CarZed},
  the unique map $\langle\eta,\mu\rangle\colon C\rightarrow R$
  defined by
   $\langle\eta,\mu\rangle\circ i_P=\eta$ and
$\langle\eta,\mu\rangle\circ i_Q=\mu$ is a \Zed-map.
   Then $C$ is the co-product of $P$ and $Q$ in the category $\QP$.
   \medskip

\noindent {\it Pushouts of strict \Zed-maps}:
   Let $P\subseteq \R^{n}$, $Q\subseteq \R^{m}$ and $D\subseteq
\R^{k}$ be polyhedra and
   $\eta\colon P\rightarrow Q$ and $\mu\colon P\rightarrow D$ be
strict $\Zed$-maps.

   With reference to    Lemma \ref{Lem_TrianofVstrictmap},
   let  $\Delta_{P}$, $\Delta_{Q}$ and $\Delta_{R}$ be regular
triangulations of $P$, $Q$ and $R$
   such that $\eta$ and $\mu$ are linear over each simplex of $\Delta_{P}$,
   $\eta(\Delta_P)$ is a full subcomplex of $\Delta_Q$,
   and $\mu(\Delta_P)$ is a full subcomplex of $\Delta_R$.

   Let $f=\eta\restrict{\ver(\Delta_P)}$ and $g=\mu\restrict{\ver(\Delta_P)}$.
   By definition,
   $f\colon \mathfrak{W}(\Delta_P)\rightarrow \mathfrak{W}(\Delta_Q)$
and $g\colon \mathfrak{W}(\Delta_P)\rightarrow
\mathfrak{W}(\Delta_R)$ are d-maps
and the following diagram commutes:
   $$
     \xymatrix{
     {}&{Q\ar[rr]^{\iota_{\Delta_Q}}}&{}&{\Pol(\mathfrak{W}_{\Delta_Q})}\\
{P\ar[ur]^{\eta}\ar[dr]^{\mu}\ar[rr]^{\iota_{\Delta_P}}}&{}&{\Pol(\mathfrak{W}_{\Delta_P})\ar[ur]^{\Pol(f)}\ar[dr]^{\Pol(g)}}&{}\\
     {}&{R\ar[rr]^{\iota_{\Delta_R}}}&{}&{\Pol(\mathfrak{W}_{\Delta_R})}\\
     }
   $$

   Since $\eta$ and $\mu$ are one-one,  then so are
   $f$ and $g$,  and we can write
   \begin{center}
   \begin{tabular}{lcl}
      $\ver(\Delta_{P})$ & $=$ & $\{v_{1},\ldots,v_{r}\}$ \\
      $\ver(\Delta_{Q})$ & $=$ &
$\{f(v_{1}),\ldots,f(v_m),w_{1},\ldots,w_{s}\}$ \\
      $\ver(\Delta_{R})$ & $=$ &
$\{g(v_{1}),\ldots,g(v_m),z_{1},\ldots,z_{t}\}$
   \end{tabular}
   \end{center}

   Let $W=\ver(\Delta_{Q})\setminus f(\ver(\Delta_{P}))$ and
$Z=\ver(\Delta_{R})\setminus g(\ver(\Delta_{P}))$.
   Without loss of generality,
     $W\cap Z=\emptyset$.

   We now define the weighted abstract simplicial complex
   $\mathfrak{W}=\langle V,\Sigma,\omega\rangle$ by the following stipulation:
   \begin{itemize}
       \item $V=\ver(\Delta_{P})\cup W\cup Z $;
       \item   $\omega(v)=\den(v)$ for each $v\in V$;
       \item $X\in\Sigma$ if
       \begin{itemize}
         \item[(a)] either  $X\subseteq \ver(\Delta_{P})\cup W$ and
$f(X\cap \ver(\Delta_{P}))\cup (X\cap W)\in\Sigma_{\Delta_{Q}}$, or
         \item[(b)]$X\subseteq \ver(\Delta_{P})\cup Z$ and $g(X\cap
\ver(\Delta_{P}))\cup (X\cap Z)\in\Sigma_{\Delta_{R}}$.
       \end{itemize}
   \end{itemize}

   Let $i_Q\colon \ver(\Delta_Q)\rightarrow V$ and
   $i_R\colon \ver(\Delta_R)\rightarrow V$ be defined as follows:
   $$
     i_Q(v)=\left\{\begin{tabular}{ll}
       $v$& if $v\in W$\\
       $w$& if $w\in \ver(\Delta_{P})$ and $v=f(w)$,
     \end{tabular}\right.
   $$
   $$
     i_R(v)=\left\{\begin{tabular}{ll}
       $v$& if $v\in Z$\\
       $w$& if $w\in \ver(\Delta_{P})$ and $v=g(w)$.
     \end{tabular}\right.
   $$
   By  definition of $\Sigma$ and $\omega$,  the
   maps
   $i_Q$ and $i_R$ preserve weights,
   whence  they are morphisms in $\wasc$.
   By Corollary \ref{Cor-PreserDen},
     $\Pol(i_Q)$ and $\Pol(i_R)$ are strict \Zed-maps and
     we have a commutative diagram

       $$
     \xymatrix{
{}&{Q\ar[rr]^{\iota_{\Delta_Q}}}&{}&{\Pol(W_{\Delta_Q})\ar[dr]^{\Pol(i_Q)}}&{}\\
{P\ar[ur]^{\eta}\ar[dr]^{\mu}\ar[rr]^{\iota_{\Delta_P}}}&{}&{\Pol(W_{\Delta_P})\ar[ur]^{\Pol(f)}\ar[dr]^{\Pol(g)}}&{}&{\Pol(W)}\\
{}&{R\ar[rr]^{\iota_{\Delta_R}}}&{}&{\Pol(W_{\Delta_R})\ar[ur]^{\Pol(i_R)}}&{}\\
     }
   $$

        \medskip

          \medskip

   Observe that $\rho_Q=\Pol(i_Q)\circ \iota_{\Delta_Q}$ and
$\rho_R=\Pol(i_R)\circ \iota_{\Delta_R}$ are strict \Zed-maps. Since
$i_Q(f(v))=i_R(g(v))=v$ for each $v\in \ver(\Delta_P)$, we have
   $\rho_Q\circ \eta=\rho_R\circ \mu=\rho_P$.

\medskip
We {\it claim} that  for all $x\in Q$ and $y\in R$ with  $\rho_Q(x)=\rho_R(y)$,
    there exists $z\in P$ such that $\eta(z)=x$ and $\mu(z)=y$.

    As a matter of fact,
     let $S\in \Sigma_{\Delta_{Q}}$
        and   $T\in \Sigma_{\Delta_{R}}$
satisfy  $x\in S$
and  $y\in T$.
Then $\rho_Q(x)=\rho_R(y)\in\rho_Q(S)\cap\rho_R(T)\in \Sigma$, whence
  $\rho_Q(S)\cap\rho_R(T)=\conv(\rho_Q(\ver(S))\cap \rho_R(\ver(T)))$.
  Our assumptions about
  $i_Q$, $i_R$ and  $W\cap Z=\emptyset$ are to the effect
  that
  $\rho_Q(\ver(S))\cap \rho_R(\ver(T))\subseteq\ver(\Delta_P)$.
   Then there exists $z\in P$ such that $\rho_P(z)=
   \rho_Q\circ \eta(z)=\rho_R\circ \mu(z)=\rho_Q(x)=\rho_R(y)$.
   Since $\rho_Q$ and $\rho_R$ are one-one, $\eta(z)=x$ and $\mu(z)=y$, as claimed.

\medskip
We have proved:
   \begin{equation}
     \rho_Q(Q)\cap\rho_R(R)=\rho_Q\circ f(P)=\rho_R\circ
g(P)=\rho_P(P).\label{Eq_Amalgam}
   \end{equation}

   In order to prove that $\Pol(\mathfrak{W})$ is the pushout $Q\coprod_P R$ in $\wasc$,
let  $U$ be a rational polyhedron, together with \Zed-maps
   $\gamma_Q\colon Q\rightarrow U$ and $\gamma_R\colon R\rightarrow U$
    such that $\gamma_Q\circ\eta=\gamma_R\circ\mu$.
Repeated applications
  of Lemma \ref{Lem_StableTriangul} provide regular triangulations $\nabla_Q$
  of $Q$ and
$\nabla_R$ of $R$   such that $\gamma_Q$ and $\rho_Q$ are linear on
each simplex of $\nabla_Q$, and $\gamma_R$ and $\rho_R$  are linear
on each simplex of $\nabla_R$.
  Lemma \ref{Lem_TrianofVstrictmap}  yields  regular triangulations
   $\Lambda_P$, $\Lambda_Q$, $\Lambda_R$ of $P$, $Q$, $R$
   such that
   \begin{itemize}
     \item[(i)] $\eta$ and $\mu$ are linear on each simplex of $\Lambda_P$;
     \item[(ii)] $\eta(\Lambda_P)\subseteq \Lambda_Q$ and
$\mu(\Lambda_P)\subseteq \Lambda_R$;
     \item[(iii)] $\Lambda_Q$, $\Lambda_R$ are subdivisions of
$\nabla_Q$, $\nabla_R$.
   \end{itemize}
Thus
  $\gamma_Q$ and $\rho_Q$ are linear on each simplex of $\Lambda_Q$,
and $\gamma_R$ and $\rho_R$  are linear on each simplex of
$\Lambda_R$.

As a consequence,
  $\rho_Q(\Lambda_Q)\cap\rho(\Lambda_R)=\rho_Q\circ\eta(\Lambda_P)=\rho_R\circ\mu(\Lambda_P)$
  is a regular triangulation of
$\rho_Q(Q)\cap\rho_R(R)=\rho_Q\circ\eta(P)=\rho_R\circ\mu(P)$.
   Thus, $\rho_Q(\Lambda_Q)\cup\rho(\Lambda_R)=\Lambda$ is a well
defined simplicial complex.
   Moreover, $\Lambda$ is a regular triangulation of $\Pol(W)$.

   Finally, let $\zeta\colon \Pol(W)\rightarrow U$ be the unique
\Zed-map given by Corollary \ref{Cor_ExtensionToZed} which is linear
over each simplex of $\Lambda$ and
   also satisfies
   $$
     \zeta(v)=\left\{\begin{tabular}{rl}
                 $\gamma_Q(x)$& if $v=\rho_Q(x)$\\
                 $\gamma_R(y)$& if $v=\rho_R(y)$,
              \end{tabular}\right.
   $$
   on each vertex $v$ of $\Lambda$.
   Since  $\rho_Q(\nabla_Q)\cap\rho(\nabla_R)=\rho_Q\circ\eta(\nabla_P)$
   then $\zeta$ is well defined.
For any vertex
$x$  of  $\nabla_Q$,
  $\gamma_Q(x)=\zeta(\rho(x))$.
  Since $\gamma$ is linear over each simplex of
  $\nabla_Q$ then $\gamma_Q=\zeta\circ\rho_Q$.
   Similarly, $\gamma_R=\zeta\circ\rho_R$,
   which proves that $\Pol(W)$ is the pushout $Q\coprod_P R$ in $\wasc$.
\end{proof}

\begin{corollary}
\label{corollary:scazonte}
Let    $(G_1,u_1)$, $(G_2,u_2)$, $(G_3,u_3)$ be
finitely presented unital $\ell$-groups
with  onto homomorphisms
$f\colon G_1\rightarrow G_3$,  $g\colon G_2\rightarrow G_3$.
Then the fiber product $G=\{(a,b)\in G_1\times G_2\mid f(a)=g(b)\}$
  (with $(u_1,u_2)$  as the distinguished order  unit)
is a finitely presented unital $\ell$-group.
\end{corollary}
\begin{proof}
The result
  follows from Theorems \ref{Theo_Baker-Beynon}, \ref{Theo_StricOnto}
and \ref{Theo-Colimits},
  upon observing that for all rational polyhedra $P,Q,R$ and strict
\Zed-maps $\eta\colon P\rightarrow Q$, $\mu\colon P\rightarrow R$ the
unital $\ell$-group $\McN(Q\coprod_P R)$ is isomorphic to the fiber
product $\{(f,g)\in \McN(Q)\times \McN(R)\mid f\circ\eta=g\circ\mu\}$.
\end{proof}

In Corollary \ref{Cor_FiberProj} we will prove that if in
Corollary \ref{corollary:scazonte}
we also assume
that $(G_1,u_1)$, $(G_2,u_2)$, and $(G_3,u_3)$ are projective, then so is
   $G$.

\section{Exact unital $\ell$-groups}

Working in the framework of   intuitionistic logic, in
  his paper \cite{DJ1982},  de Jongh  calls  ``exact'' a formula
   $\varphi$ such that the Heyting algebra
   presented by $\varphi$ is embeddable into a free Heyting algebra
(see also Section \ref{Sec_admissible}).

   Accordingly, in this paper we say that
a  unital $\ell$-group $(G,u)$ is   {\it exact} if it is finitely presented
and there exist a positive integer $n$ and a one-one unital
$\ell$-homomorphism $g $ of $(G,u)$ into the free
unital $\ell$-group  $\McN([0,1]^{n})$. By Lemma \ref{Lem_ImagesZmorph0} and Theorem \ref{Theo_Baker-Beynon},  a unital $\ell$-group is exact iff it is isomorphic to a finitely generated unital $\ell$-subgroup of $\McNn$. This equivalent definition can also be obtained as an application of the equivalence between MV-algebras and unital $\ell$-groups and
\cite[Corollary 6.6]{Mu2011}.

In this section we will give a characterization of exact unital $\ell$-groups.

A unital $\ell$-group $(G,u)$ is {\it projective} if whenever
$\psi\colon (G_1,u_1)\rightarrow(G_2,u_2)$ is a unital
$\ell$-homomorphism onto $(G_2,u_2)$
and $\phi\colon (G,u)\to(G_2,u_2)$ is a unital $\ell$-homomorphism, there
is a unital $\ell$-homomorphism $\theta\colon (G,u)\to(G_1,u_1)$ such that
$\phi= \psi \circ \theta$.

A  finitely generated unital $\ell$-group $(G,u)$ is projective iff it is
a {\it retraction} of  the free unital $\ell$-group  $\McNn$ for some
$n\in\{1,2,\ldots\}$.
In other words, there is a homomorphism $\iota\colon (G,u)\to \McNn$ and a
homomorphism $\sigma\colon \McNn \to (G,u)$ such that
$\sigma\circ\iota$ is the identity map $Id_{G}$
on $G$.

Lemma \ref{Lem_ImagesZmorph} yields the following inclusions
for unital $\ell$-groups:
\begin{equation}
\label{Eq_Inclusions}
\mbox{Finitely Presented }\supseteq\mbox{ Exact }\supseteq\mbox{
Finitely generated projective}.
\end{equation}

As shown in
  \cite{CM20XX},  the class of finitely generated projective unital
$\ell$-groups (includes but)
   does not coincide with the class of finitely presented unital $\ell$-groups.

   An example of   an exact unital
$\ell$-group which is  not projective is as follows:

\begin{example}
  Let $P=\{(x,y)\in[0,1]^{2}\mid x\in\{0,1\}\mbox{ or } y\in\{0,1\}\}$.
  The map $\eta\colon [0,1]\rightarrow P$ defined by
$$
\eta(a)=\left\{\begin{tabular}{ll}
          $(3x,0)$& if $0\leq a\leq 1/3$,\\
          $(1,6x-2)$& if $1/3\leq a\leq {1/2}$,\\
          $(4-6x,1)$& if $1/2\leq a\leq {2}/{3}$,\\
          $(0,3-3x)$& if ${2}/{3}\leq a\leq 1$,
         \end{tabular}\right.
$$
is a \Zed-map onto $P$. By Theorem \ref{Theo_monicepi},
$\McN(\eta)\colon\McN(P)\rightarrow \McN([0,1])$ is one-one.
Then $\McN(P)$ is exact. Since $P$ is not simply connected, by
  \cite[Theorem 4.2]{CM20XX}   $\McN(P)$ is not projective.
\end{example}
As is well known (see \cite{Be1977A}),
 for $\ell$-groups the three notions in
(\ref{Eq_Inclusions})  coincide.

\subsection{Strongly regular triangulations}
The notion of {\it strongly regular triangulation} was introduced in \cite[Definition 3.1]{CM20XX} and it has a key role in our characterization of exact unital $\ell$-groups.
\begin{definition}\label{definition:strongly}
   A simplex $S$ is said to  be  {\it strongly regular} if it is regular
and the greatest common divisor of the denominators of the vertices of $S$
  is equal to
$1$.
   A   triangulation $\Delta$  of a polyhedron $P\subseteq \cube$ is
said to be {\it strongly regular}
   if each maximal simplex of $\Delta$  is strongly regular.
\end{definition}

\begin{lemma}\label{Lem_subdiv preserves den}
   Let $S$ be a regular k-simplex.
   Then for every regular k-simplex $T$  such that $T\subseteq S$,
   the greatest common divisor of the denominators of the vertices of $T$
   is equal to the greatest common divisor of the denominators of the
vertices of
   $S$.
\end{lemma}

\begin{proof}
   In view of Lemma \ref{lemma:weakoda} it is no loss of
   generality to assume that
    $T$ is  one of the maximal simplexes obtained by blowing up
    $S$ at the Farey mediant $v$ of its vertices.
   Let $v_1,\ldots,v_k$ be the vertices of $S$.
   Since $\den(v)$ is equal to  $\sum_{j=0}^k \den(v_j)$, then
     for each $i=1,\ldots,k$
     the
      greatest common divisor of  the integers
$\den(v_1),\ldots,\den(v_k)$ coincides with the greatest common
divisor of the set of integers
   $$
    \{ \den(v_0),\ldots,\den(v_{i-1}), \den(v),\den(v_{i+1}),\ldots,\den(v_k)\}
   $$
Thus the
greatest common divisor of the denominators of the vertices $T$
   is equal to the greatest common divisor of the denominators of the
vertices $S$.
\end{proof}

The following result was first proved in \cite[Lemma 3.2]{CM20XX}:

\begin{corollary}
   \label{Cor-BlowupPreserGCD}
   Let  $\Delta$ and $\nabla$ be regular triangulations of a
polyhedron $P\subseteq\cube$.
   Then $\Delta$ is strongly regular iff so is $\nabla$.
\end{corollary}

Since strong regularity does not depend on the regular triangulation
$\Delta$ of $P$,
without fear of ambiguity we may say that  a polyhedron $P$
  is {\it strongly regular} if some (equivalently,  each) of its regular
triangulations
   is strongly regular.

\begin{example}\label{Ex_CubeStronglyTriang}
   The $n$-cube $\cube$ is strongly regular.
   To see this, let us equip the set $\{0,1\}^{n}$  with the following
partial order $(a_1,\ldots,a_1)\leq (b_1,\ldots,b_n)$ iff
$a_i\leq b_i$ for each $i=1,\ldots,n$.
   Let $\Delta$ be the triangulation of $\cube$ formed by the
simplexes $\conv(C)$ whenever $C$ is a chain in the poset
$(\{0,1\}^{n},\leq)$.
   This is called the {\em standard triangulation} of the cube in
\cite[p. 60]{Se1982}.
   Since the denominator of every vertex of $\Delta$ is $1$, it
follows that $\Delta$ is strongly regular. The desired conclusion now
follows from Corollary \ref{Cor-BlowupPreserGCD}.
\end{example}

For all
 $v,w\in\R^{n}$, we let $\dist(v,w)$ denote
  their Euclidean distance in $\R^{n}$.
For each $0\leq\delta\in\R$ and $v\in\R^{n}$, we use
the notation
$B(\delta,v)=\{w\in\R^{n}\mid \dist(v,w)<\delta\}$.
The
dimension
of the ambient space will always be clear from the context.

\medskip
{}From
  Lemma \ref{Lem_subdiv preserves den}
  we obtain the following characterization of strong regularity:
\begin{corollary}\label{Cor_SRandCoprvectors}
   Let $P\subseteq\R^{n}$ be a polyhedron. Then the following
   conditions  are equivalent:
   \begin{itemize}
   \item[(i)] $P$ is strongly regular.
   \item[(ii)] For each $v\in P$ and each $0<\delta\in\R$, there exists
   $w\in P$ such that $\dist(v,w)<\delta,$ and $\den(v)$ and $\den(w)$ are
   {\em coprime}, in the sense that ${\rm gcd}(\den(v),\den(w))=1$.
   \end{itemize}
\end{corollary}

For any set $T\subseteq\R^{n}$, we let
${\rm aff}(T)$ denote the {\it affine hull} of $T$,  i.e.
 $${\rm
aff}(T)=\left\{\sum_{i=0}^{m}\lambda_i v_i\mid \mbox{for some }v_i\subseteq T, \lambda_i\in\R, \sum_{i=0}^{m}\lambda_i=1\mbox{ and }m=1,2,\ldots
\right\}.
$$
Further,   ${\rm relint}(T)$ denotes the
relative interior of $T$, that is, the interior of $T$ in the relative
topology of  ${\rm aff}(T)$.

\bigskip
For later use in the proofs of
Theorems \ref{Theo_StrongPreserved}
and \ref{Thm_AnchvsStReg}, we
record here the following elementary
characterization:

\begin{lemma}\label{Lem_MaxSimpTriang}
   Let $\Delta$ be a triangulation of a polyhedron $P\subseteq \R^{n}$
and $T\in \Delta$. Then the following conditions are equivalent:
   \begin{itemize}
    \item[(i)] $T$ is maximal in $\Delta$.
    \item[(ii)] Whenever    $w\in{\rm relint}(T)$ and
      $v\in\R^{n}$ does not lie in the affine hull
    of  $T$, then $\conv(v,w)$ is not contained in $P$.
    \item[(iii)] For every $v\in{\rm relint}(T)$ there exists
$0<\delta\in\R$ such that $B(\delta,v)\cap P\subseteq T$.
   \end{itemize}
\end{lemma}

\begin{theorem}\label{Theo_StrongPreserved}
   Let $P$ and $Q$ be polyhedra and $\eta\colon P\rightarrow Q$
   be a \Zed-map onto $Q$.
   If $P$ has a strongly regular triangulation then $Q$ has a strongly
regular triangulation.
\end{theorem}

\begin{proof}
   Let $\Delta$ be a regular triangulation of $Q$ and $S$  a maximal
simplex of $\Delta$.
   Let $d$ denote the greatest common divisor of the denominators of the vertices of $S$.
   Let $v\in {\rm relint}(S)$. By
    Lemma \ref{Lem_MaxSimpTriang} there exists $0<\epsilon\in\R$ such
that $B(\epsilon,v)\cap Q\subseteq S$.
    Since $\eta$ is a continuous onto map,
     there exist $w\in P$ and  $0<\delta\in\R$ such that
$\eta(B(\delta,w)\cap P)\subseteq B(\epsilon,v)$.

Since $P$ is strongly regular,
 Corollary \ref{Cor_SRandCoprvectors}
yields  $z\in B(\delta,w)\cap P$ such that $\den(w)$ and $\den(z)$
   are  coprime. Then $\eta(w)=v,\eta(z)\in S$,
and  $d$ is a common divisor of
 $\den(\eta(w))$ and $\den(\eta(z))$.
   By Corollary \ref{Cor_Div_Denominators}, $d$ is a
   common  divisor of
   $\den(w)$  and $\den(z)$.
We conclude that  $d=1$.
\end{proof}

\subsection{Finitely generated projective unital $\ell$-groups}
We now
collect some definitions and results from \cite{CM20XX} that will be
necessary for the geometrical description of exact unital
$\ell$-groups in Theorem \ref{Theo_WeakProj}.
In
Theorem \ref{Thm_AmalProj} it is also proved that
projectiveness is preserved under fiber products of onto homomorphisms
of unital $\ell$-groups.

A $\Zed$-map $\sigma\colon P\rightarrow P$ is a
{\it \Zed-retraction of P} if $\sigma\circ\sigma = \sigma$.
The rational polyhedron $R=\sigma(P)$  is said to be a
\Zed{\it-retract of $P$}.
A rational polyhedron $Q$ is said to be a {\it \Zed-retract} if it is
a \Zed-retract of $[0,1]^{n}$ for some $n\in\{1,2,\ldots\}$.

The following is a consequence of Theorem \ref{Theo_Baker-Beynon}
(see \cite[Theorem 1.2]{CM2009} for details):

\begin{theorem}
   A  unital $\ell$-group $(G,u)$ is finitely generated projective iff   it is isomorphic to $\McN(P)$ for some \Zed-retract $P$.
\end{theorem}

A  simplex $T\in\Sigma$
of
an abstract simplicial complex $\langle V,\Sigma\rangle$
  is said to have a {\it free face} $F$
if
\begin{itemize}
\item $\emptyset\neq F\subseteq T$ is a {\it facet} (=maximal proper subset)
  of $T$, and
\item whenever  $F\subseteq S\in\Sigma$ then $S=F$ or $S=T$.
\end{itemize}
It follows that $T$ is a maximal simplex of $\Sigma$, and the removal
from $\Sigma$ of both $T$ and $F$ results in the subcomplex $\langle
V',\Sigma'=\Sigma\setminus\{T,F\}\rangle$ of  $\langle
V,\Sigma\rangle$, where $V'=V\setminus F$ if $F$ is a singleton and
otherwise $V'=V$.
The transition  from  $\langle V,\Sigma\rangle$ to  $\langle
V',\Sigma'\rangle$ is called an {\it (abstract) elementary collapse}.
If a simplicial complex $\langle W,\Gamma\rangle$ can be obtained
from  $\langle V,\Sigma\rangle$ by a sequence of elementary collapses
we say that  $\langle V,\Sigma\rangle$ {\it collapses to}  $\langle
W,\Gamma\rangle$.
We say that
the simplicial complex
$\langle V,\Sigma\rangle$ is {\it collapsible} if it
collapses to the abstract simplicial complex consisting of one of its
vertices (equivalently,
 any of its vertices \cite[p.248]{Wh1939}).

See \cite[\S III, Definition 7.2]{Ew1996}, \cite[p.247]{Wh1939} for the geometrical
counterpart of collapsibility.
For the purposes of this paper it is enough to observe that a regular triangulation
$\Delta$ is collapsible iff its skeleton
$\mathfrak{W}(\Delta)$ is collapsible.

\begin{theorem}\cite[Theorem 6.1]{CM20XX}\label{Thm_collapsible}
   Let $P\subseteq [0,1]^{n}$ be a  polyhedron. Suppose
\begin{itemize}
   \item[(i)] $P$ has a collapsible  triangulation $\nabla$;
   \item[(ii)] $P$ contains a vertex $v$ of $\cube$;
   \item[(iii)] $P$ is strongly regular.
\end{itemize}
Then  $P$ is a $\Zed$-retract of $[0,1]^{n}$.
\end{theorem}

The following result states
  that \Zed-retracts are preserved under pushouts of strict \Zed-maps.

\begin{theorem}\label{Thm_AmalProj}
  Let $P\subseteq [0,1]^{n}$, $Q\subseteq  [0,1]^{m}$ and $R\subseteq
[0,1]^{k}$ be
   \Zed-retracts and
   $\eta\colon P\rightarrow Q$ and $\mu\colon P\rightarrow D$ be
strict $\Zed$-maps. Then the pushout $Q\coprod_P R$ (whose existence
is ensured by Theorem \ref{Theo-Colimits}) is a \Zed-retract.
\end{theorem}
\begin{proof}

With the notation of the proof of  Theorem \ref{Theo-Colimits}, the
pushout $Q\coprod_P R$ was realized therein as the rational
polyhedron $\Pol(\mathfrak{W})\subseteq[0,1]^{r+s+t}$ of a
certain weighted abstract simplicial complex $\mathfrak{W}$.
The embeddings  of $P$, $Q$, $R$ into $Q\coprod_P R$,
were  denoted $\rho_P$,
$\rho_Q$, and $\rho_R$.
Letting $A=[0,1]^{r+s}\times\{0\}^{t}$ and
$B=[0,1]^{r}\times\{0\}^{s}\times[0,1]^{t}$,
it was shown:

\begin{itemize}
  \item[(i)]
$\rho_Q(Q)\subseteq A$;

\smallskip
  \item[(ii)]
$\rho_R(R)\subseteq A$;

\smallskip
  \item[(iii)] $\rho_Q(Q)\cup\rho_R(R)=Q\coprod_P R$;

  \smallskip
  \item[(iv)] $\rho_Q(Q)\cap\rho_R(R)=\rho_P(P)$.
\end{itemize}

\medskip
\noindent
Since $P$, $Q$, and $R$ are \Zed-retracts, by \cite[Lemma
4.2]{CM20XX} $\rho_P(P)$, $\rho_Q(Q)$, and $\rho_R(R)$ are
\Zed-retracts. We let  $\gamma_P\colon[0,1]^{r+s+t}\rightarrow
\rho_P(P)$, $\gamma_Q\colon[0,1]^{r+s+t}\rightarrow \rho_Q(Q)$  and
$\gamma_R\colon[0,1]^{r+s+t}\rightarrow \rho_R(R)$ denote the
corresponding \Zed-retractions.

\medskip
The \Zed-retraction for $Q\coprod_P R$ will be constructed in three steps.
\medskip

\noindent{\it Step 1:}

If $x\in A\cup B$, then $\conv({\bf 0},x)\subseteq A\cup B$, where
${\bf 0}$ denotes the origin of $\R^{r+s+t}$.  Therefore, by
\cite[Theorem 1.4]{CM2009}, there exists a \Zed-retraction
$\gamma_1\colon[0,1]^{r+s+t}\rightarrow A\cup B$ onto $A\cup B$.
\medskip

\noindent{\it Step 2:}

Combining Proposition \ref{proposition:poly} and Lemma
\ref{Lem_Triang-Subset}, there is a regular triangulation $\Delta$ of
$A\cup B$, such that $\Delta_{Q\coprod_P R}=\{S\in \Delta\mid
S\subseteq Q\coprod_P R\}$ is a full subcomplex of $\Delta$.  By Lemma
\ref{Cor_ExtensionToZed} there exists a unique \Zed-map
$\gamma_2\colon A\cup B\rightarrow A\cup B$ satisfying
$$
   \gamma_2(v)=\left\{\begin{tabular}{ll}
      $v$& if $v\in Q\coprod_P R$,\\
      $\gamma_P(v)$& otherwise,
   \end{tabular}\right.
$$
with $\gamma_2 $ linear over each simplex $S\in \Delta$.
By construction,  $\gamma_2$ has the following properties:
\begin{itemize}
  \item[(a)] if $v\in  Q\coprod_P R$, then $\gamma_2(v)=v$;
  \item[(b)] if $\gamma_2(A)\cap \gamma_2(B)=\rho_P(P)$
\end{itemize}
\medskip

\noindent{\it Step 3:}

By (a), we have $Q\coprod_P R\subseteq \gamma_2(A\cup B)$. Again by Proposition
\ref{proposition:poly} and Lemma \ref{Lem_Triang-Subset}, there is a
regular triangulation $\Lambda$ of $\gamma_2(A\cup B)$, such that
$\Lambda_{Q\coprod_P R}=\{S\in \Lambda\mid S\subseteq Q\coprod_P R\}$
is a full subcomplex of $\Lambda$. Let $\gamma_3\colon \gamma_2(A\cup
B)\rightarrow Q\coprod_P R$ be defined by:
$$
   \gamma_3(v)=\left\{\begin{tabular}{ll}
      $\gamma_Q(v)$& if $v\in \gamma_2(A)$,\\
      $\gamma_R(v)$& if $v\in \gamma_2(B)$.
   \end{tabular}\right.
$$

By (b) and (iv), if $v\in \gamma_2(A)\cap\gamma_2(B)$,  $v\in
\rho_P(P)=\rho_Q(Q)\cap\rho_R(R)$, which implies that $
\gamma_Q(v)=\gamma_R(v)=v$.
This shows that  $\gamma_3$ is well defined.

\bigskip

We {\it claim} that the map
$\gamma=\gamma_3\circ\gamma_2\circ\gamma_1\colon[0,1]^{r+s+t}\rightarrow
Q\coprod_P R$ is a \Zed-retraction.
If $v\in Q\coprod_P R$, by definition of $\gamma_1$, $\gamma_2$ and
$\gamma_3$ it follows that
$\gamma_1(v)=\gamma_2(v)=\gamma_3(v)=v$.
If $v\notin Q\coprod_P R$, then $\gamma_1(v)\in A\cup B$. Assume
first   $\gamma_1(v)\in A$. Then $\gamma_2\circ\gamma_1(v)\in\gamma_2(A)$. Further,
$\gamma(v)=\gamma_3(\gamma_2\circ\gamma_1(v))=\gamma_Q(v)\in\rho_Q(Q)\subseteq
Q\coprod_P R$. Similarly, if $\gamma_1(v)\in B$ then
$\gamma(v)\in Q\coprod_P R$. Therefore, $\gamma[0,1]^{r+s+t}\rightarrow Q\coprod_P R$ is a
\Zed-retraction onto $Q\coprod_P R$, as claimed.

\medskip
The proof is complete
  \end{proof}

Combining  Corollary \ref{corollary:scazonte}
with the foregoing theorem
we get:

\begin{corollary}\label{Cor_FiberProj}
    Let $(G_1,u_1)$, $(G_2,u_2)$,  $(G_3,u_3)$ be finitely
generated projective unital $\ell$-groups
     and  $f\colon G_1\rightarrow G_3$ and $g\colon G_2\rightarrow
G_3$ onto homomorphisms.
  Then the fiber product $G=\{(a,b)\in G_1\times G_2\mid f(a)=g(b)\}$
  (with $(u_1,u_2)$ as the distinguished order unit)
   is a finitely generated projective unital $\ell$-group.
\end{corollary}

\subsection{Geometric realization of exact unital $\ell$-groups}

The rest of the section is devoted  to proving
\begin{theorem}\label{Theo_WeakProj}
   A unital $\ell$-group $(G,u)$ is exact iff there exists
a polyhedron $P\subseteq\R^{n}$ satisfying the following conditions:
   \begin{enumerate}
     \item $(G,u)\cong\McN(P)$;
     \item $P$ is connected;
     \item $P\cap\Zed^{n}\neq\emptyset$;
     \item $P$ is strongly regular.
   \end{enumerate}
\end{theorem}

For the proof we prepare:

\begin{lemma}\label{Lem_CarWeakProj}
   A unital $\ell$-group $(G,u)$ is exact iff there exist integers
    $m,n\geq 0$, and a \Zed-map $\eta\colon [0,1]^{n}\rightarrow
\R^{m}$ such that $(G,u)\cong \McN(\eta([0,1]^{n}))$.
\end{lemma}
\begin{proof}

$(\Rightarrow)$
For some $m,n\in\{1,2,\ldots\}$ there
  exist unital $\ell$-ho\-mo\-mor\-phisms
   $$f\colon\McN([0,1]^{m})\rightarrow (G,u)
   \,\,\,\mbox{and}\,\,\,
   g\colon(G,u)\rightarrow \McN([0,1]^{n})$$
with
  $f$  onto $(G,u)$ and $g$  one-one.
   Then $$(G,u)\cong\McN([0,1]^{m})/{\rm ker}f= \McN([0,1]^{m})/{\rm
   ker}(f\circ g).$$
   Theorem \ref{Theo_Baker-Beynon} yields
   a \Zed-map $\eta\colon[0,1]^{n}\rightarrow [0,1]^{m}$ such that
$g\circ f=\McN(\eta)$.
   Therefore, $h\in{\rm ker}(f\circ g)$ iff $h\circ\eta=0$ iff
$h(\eta([0,1]^{n}))=\{0\}$,
   whence $\McN([0,1]^{m})/{\rm ker}(f\circ g)=\McN([0,1]^{m})/{\rm
ker}(\McN(\eta))=\McN(\eta([0,1]^{n}))$.
\smallskip

$(\Leftarrow)$
As  observed in Section 2,
(see  \eqref{Eq:IsoinCube} in particular)
every polyhedron is \Zed-homeomorphic to a
polyhedron contained in some unit cube $[0,1]^{m}$.
Thus without loss of generality we can
assume $\eta([0,1]^{n})\subseteq [0,1]^{m}$.  Let $\eta\colon
[0,1]^{n}\rightarrow [0,1]^{m}$ be a \Zed-map such that $(G,u)\cong
\McN(\eta([0,1]^{n}))$.  Let further
$\mu\colon\eta([0,1]^{n})\rightarrow [0,1]^{m}$ and $\nu\colon
[0,1]^{n}\rightarrow \eta([0,1]^{n}) $ respectively be a strict and an
onto \Zed-map such that $\eta=\mu\circ\nu$.
By Theorems \ref{Theo_monicepi} and \ref{Theo_StricOnto},
$\McN(\mu)\colon \McN([0,1]^{m})\rightarrow \McN(\eta([0,1]^{n})) $ is
an onto unital $\ell$-homomorphism and $\McN(\nu)\colon
\McN(\eta([0,1]^{n}))\rightarrow \McN([0,1]^{n})$ is a one-one unital
$\ell$-homomorphism.
   Since $\McN(\eta([0,1]^{n}))\cong(G,u)$,  $(G,u)$ is finitely
generated and is
    isomorphic to a subalgebra of the free unital $\ell$-group
$\McN([0,1]^{n})$.
\end{proof}

\begin{lemma}\label{Lem_CarZimages}
   Let $P\subseteq\R^{n}$ be a polyhedron.
   Then for some  $l=1,2,\ldots$ there is a  \Zed-map $\eta$
   of $[0,1]^{l}$ onto $P$
  iff $P$ satisfies the following three conditions:
   \begin{enumerate}
     \item $P$ is connected;
     \item $P\cap\Zed^{n}\neq\emptyset$;
     \item $P$ is strongly regular.
   \end{enumerate}
\end{lemma}
\begin{proof}
$(\Rightarrow)$  If $\eta\colon [0,1]^{l}\rightarrow P$ is an onto
\Zed-map, then $P$ is connected because $\eta$ is continuous.
   Combining Example \ref{Ex_CubeStronglyTriang} and Theorem
\ref{Theo_StrongPreserved}, it follows that $P$ is strongly regular.
   By Corollary \ref{Cor_Div_Denominators}, $\den(\eta(0,\ldots,0))$
is a divisor of $\den(0,\ldots,0)$, that is,
$\den(\eta(0,\ldots,0))=1$. Then $\eta(0,\ldots,0)\in\Zed^{n}$.

\medskip

  $(\Leftarrow)$
  For some suitable  strongly regular collapsible triangulation $\nabla$,
  we will define onto   \Zed-maps  $\eta_1\colon\cube \rightarrow |\nabla|$,
   and $\eta_2\colon |\nabla|\rightarrow P$ providing the required $\eta.$

   \medskip

   \noindent{\it Construction of $\nabla$:}

   Let $\Delta$ be a regular triangulation of $P$.
  $\nabla$ will be defined as the geometric realization of
   a
   weighted abstract simplicial complex $\mathfrak{W}$ arising
   from $\Delta$.

   Since $P$ is connected, the simple graph $H$ given by the
    $1$-simplexes of $\Delta$ is connected.
As is well known, a  {\em spanning tree}
 of $H$ is a tree
     $\mathcal{T}\subseteq\Delta$ such that
$\ver(\mathcal{T})=\ver(\Delta)=\{v_1,\ldots,v_n\}$.
   By (ii), there is no loss of generality to  assume
\begin{equation}\label{Eq:AsV1}
\den(v_1)=1.
\end{equation}

   \bigskip

\noindent
   {\it Vertices of $\mathfrak{W}$:}
   Let us set
   $$J=\{(i,j)\in\{1,\ldots,n\}^{2}\mid i\neq j \mbox{ and
}\conv(v_i,v_j)\in\Delta\}.$$
   For each $i\neq j$ such that $\conv(v_i,v_j)\in\mathcal{T}$ let
$S_{i,j}$ be a maximal simplex in $\Delta$ such that
$\conv(v_i,v_j)\subseteq S_{i,j}$.
   Let $K=\{(i,j,k)\in\{1,\ldots,n\}^{3}\mid i\neq j, j\neq k,i\neq k
\mbox{ and }\conv(v_i,v_j,v_k)\subseteq S_{i,j}\}$.
   Then
   $$
     V=\{i,\ldots,n\}\cup J\cup K.
   $$

   \bigskip

\noindent
   {\it Simplexes of $\mathfrak{W}$:} For each $i\in\{1,\ldots,n\}$,
   let the set $\mathcal{F}_i\subseteq\mathcal{P}(V)$  be defined by:
   $X\in\mathcal{F}_i$ if there are $j_1,\ldots,j_m\in\{1,\ldots,n\}$
such that $\conv(v_i,v_{j_1},\ldots, v_{j_m})\in\Delta$ and
   $X\subseteq \{i,(i,j_i),\ldots,(i,j_m)\}$.
   For each $i,j\in\{1,\ldots,n\}$ such that $i\neq j$ and
$\conv(v_i,v_j)\in\mathcal{T}$,
   we further define $\mathcal{B}_{i,j}\subseteq\mathcal{P}(V)$ as follows:
   $X\in \mathcal{B}_{i,j}$ if
   $\conv(v_i,v_j,v_{k_1}\ldots, v_{k_m})=S_{i,j}$ and
$X\subseteq\{i,j,(i,j,k_1),\ldots,(i,j,k_m)\}$.
We next let
   $$\Sigma=\bigcup\mathcal{F}_i\cup\bigcup \mathcal{B}_{i,j}$$
   By definition, if $X\subseteq Y$ and $Y\in\Sigma$ then $X\in\Sigma$.
   Moreover, for all  $x\in V$  there exists $X\in\Sigma$ such that $x\in X$.
   Therefore, $\langle V,\Sigma\rangle$ is an abstract simplicial complex.

   \bigskip

\noindent
   {\it Weights:} Finally we define $w\colon V\rightarrow
\{1,2,\ldots\}$ as follows:
   \begin{equation}\label{Eq_DefWeight}
   \begin{tabular}{lcl}
     $w(i)$     & $=$ & $\den(v_i)$,\\
     $w(i,j)$   & $=$ & $\den(v_j)$, \\
     $w(i,j,k)$ & $=$ & $\den(v_k)$.
   \end{tabular}
   \end{equation}
\medskip

\noindent{\it Claim 1:} For every maximal simplex $X$ in $\Sigma$,
  the greatest common divisor of the denominators of the vertices of $X$
     is $1$.

    By definition,
    we either have
     $$X=\{i\}\cup\{(i,j_1),\ldots,(i,j_m)\}
     \,\,\,\mbox{ and $\conv(v_i,v_{j_1},\ldots, v_{j_m})$ is maximal in $\Delta$},$$
or
      $$X=\{i,j,(i,j,k_1),\ldots,(i,j,k_m)\}\,\,\,
      \mbox{and
$\conv(v_i,v_j,v_{k_1}\ldots, v_{k_m})=S_{i,j}$}.$$

    \bigskip
    \noindent
    In either case
    the claim follows by definition of $w$,
    because   $\Delta$ is strongly regular.

    \bigskip

\noindent{\it Claim 2:}   $\mathfrak{W}$ is collapsible.

   By  definition of $\mathfrak{W}$, we have:
   \begin{itemize}
   \item[(a)] $\mathcal{F}_i\cap \mathcal{F}_j=\{\emptyset\}$ whenever
$i\neq j$;
   \item[(b)] $\mathcal{F}_{i}\cap \mathcal{B}_{j,k}=\{\emptyset\}$
whenever $i\neq j$ and $i\neq k$;
   \item[(c)] $\mathcal{F}_{i}\cap
\mathcal{B}_{j,k}=\{\emptyset,\{i\}\}$ whenever $i= j$ or $i=k$;
   \item[(d)] $\mathcal{B}_{i,j}\cap
\mathcal{B}_{k,l}=\{\emptyset,\{i\}\}$ whenever $i=k$ or $i=l$;
   \item[(e)] $\mathcal{B}_{i,j}\cap \mathcal{B}_{k,l}=\{\emptyset\}$
whenever $\{i,j\}\cap\{k,l\}=\emptyset$.
   \end{itemize}

  The claim is now proved  in 3 steps as follows:

\medskip
   {\it Step 1 ($\mathcal{F}_{i}$):}
For each $i\in\{1,\ldots,n\}$,
     $\langle\ver(\mathcal{F}_i),\mathcal{F}_i\rangle$ is
combinatorially isomorphic to
     the closed star of $v_i$ (see \cite[\S III Definition 1.11]{Ew1996}), and
therefore, it is
      a collapsible abstract simplicial complex.
     Then $\langle\ver(\mathcal{F}_i),\mathcal{F}_i\rangle$ collapses
to the vertex $i$.
By (a-e),  $\mathfrak{W}$ collapses to $\langle
V,\bigcup\mathcal{B}_{i,j}\rangle$.

\medskip
   {\it Step 2 ($\mathcal{B}_{i,j}$):}
     For each $i,j\in\{1,\ldots,2\}$
    such that $i\neq j$ and $\conv(v_i,v_j)\subseteq \mathcal{T}$, the
abstract simplicial complex
$(\ver(\mathcal{B}_{i,j}),\mathcal{B}_{i,j})$ is combinatorially
isomorphic to the skeleton of the complex given  by the simplex
$S_{i,j}$ and its faces.
Therefore,
  $\langle \ver(\mathcal{B}_{i,j}),\mathcal{B}_{i,j}\rangle$ can be
collapsed to any of its faces.
  In particular, it can be collapsed to
$\langle\{i,j\},\{\emptyset,\{i\},\{j\},\{i,j\}\}\rangle$.
Using (d) we see that
  $(V,\bigcup\mathcal{B}_{i,j})$ collapses to the abstract simplicial
complex $(V,\Sigma')$ where $X\in\Sigma'$ iff
$X\subseteq\{i,j\}$ and $\conv(v_i,v_j)\subseteq\mathcal{T}$.

\medskip
   {\it Step 3:}
The sequence of collapses defined in Steps
  1 and 2 leads to an abstract simplicial complex $(V,\Sigma')$ which
is combinatorially isomorphic to the skeleton of $\mathcal{T}$.
     Since  $\mathcal{T}$ is a tree, it is collapsible and therefore
$(V,\Sigma')$ is collapsible, too.

\bigskip

Thus
 $\mathfrak{W}$ is collapsible, and  Claim 2  is settled.
   \medskip

   By (\ref{Eq_DefWeight}) and (\ref{Eq:AsV1}), $w(1)=\den(v_1)=1$.
   {}From  Claims 1 and 2 it follows that
    $\Pol(\mathfrak{W})$ satisfies  the hypotheses  of Theorem
\ref{Thm_collapsible}.
   Therefore, for some integer  $l>0$
   there is an onto \Zed-map $\eta\colon
[0,1]^{l}\rightarrow\Pol(\mathfrak{W})$.
   \medskip

   Finally let $f\colon V\rightarrow \ver(\Delta)$ be defined as follows:
   $$
     \begin{tabular}{lcl}
       $f(i)$     & $=$ & $v_i$,\\
       $f(i,j)$   & $=$ & $v_j$, \\
       $f(i,j,k)$ & $=$ & $v_k$.
     \end{tabular}
   $$

   By (\ref{Eq_DefWeight}), $\den(f(x))=w(x)$ for each $x\in V$. By
    definition of $\mathcal{F}_i$ and $\mathcal{B}_{i,j}$,  $\,\,\,f$
is a morphism from $\mathfrak{W}$ into the skeleton
$\mathfrak{W}(\Delta)$ of $\Delta$.
   Then $\Pol(f)\colon\Pol(\mathfrak{W})\rightarrow
\Pol(\mathfrak{W}(\Delta))$ is a \Zed-map.
   Since for every  $m$-simplex $S=\conv(v_{i_0},\ldots,v_{i_m})
   \in \Delta$ there is $X\in\Sigma$ (specifically,
 $X\in \mathcal{F}_{i_0}$)
such that $f(X)=\{v_{i_0},\ldots,v_{i_m}\}$, it follows that
$\Pol(f)$ is onto $\Pol(W_\Delta)$.

\medskip
In conclusion,
$\iota_{\Delta}^{-1}\circ\Pol(f)\circ\eta\colon[0,1]^{m}\rightarrow
P$ is the desired \Zed-map onto $P$.
\end{proof}

\subsubsection*{Proof of Theorem \ref{Theo_WeakProj}}
This immediately follows from
  Lemmas \ref{Lem_CarWeakProj} and \ref{Lem_CarZimages}.

\subsection{Intrinsic characterization of exact unital $\ell$-groups}

In \cite[Definition 2.1]{MMM2006}, {\it abstract Schauder basis} were defined
for abelian $\ell$-groups as isomorphic copies of Schauder basis.
In \cite[Theorem 3.1]{MMM2006} a characterization of abstract Schauder bases is presented.
Using this characterization, in \cite[Definition 4.3]{MM2007},
the notion of abstract Schauder basis was extended to unital $\ell$-groups and called {\it basis}.
In  \cite[Theorem 4.5]{MM2007}
it is proved that an {\it archimedean} unital $\ell$-group $(G,u)$
is  finitely presented iff it has a basis.
In \cite[Theorem 3.1]{CM2011},
  the archimedean assumption was shown to be unnecessary.
  Using this latter result in combination with
   Theorem \ref{Theo_WeakProj},
will
provide in Theorem \ref{Theo-BasisExact} an algebraic description of exact  unital $\ell$-groups.

We first need to recall some definitions.
We denote by ${\rm maxspec}(G,u)$  the set of
maximal  ideals of $(G,u)$ equipped with the {\it spectral} topology:
a basis of closed sets for ${\rm maxspec}(G,u)$
is given by sets of the form
$ \{\mathfrak p \in {\rm maxspec}(G,u) \mid g\in \mathfrak p\},$
where $g$ ranges over all elements of $G$ (see \cite[\S 10]{BKW1977}).
As is well known,
${\rm maxspec} (G,u)$ is a nonempty compact Hausdorff space, \cite[Theorem 10.2.2]{BKW1977}.

\begin{definition}\label{def:basis}\cite[Definition 4.3]{MM2007}
  Let $(G,u)$ be a unital $\ell$-group.
  A {\em basis} of $(G,u)$ is a finite set $\mathcal B = \{b_{1},\ldots,b_{n}\}$
  of  elements  $\not=0$ of the positive cone $G^{+}=\{g\in G \mid g\geq 0\}$ such that
  \begin{itemize}
    \item[(i)] $\,\mathcal B$ generates $G$ using the
      group and lattice operations;
    \item[(ii)] for each $k=1,2,\ldots$ and
       $k$-element subset $C$ of $\mathcal B$
       with  $0\not= \bigwedge\{b\mid b\in C\}$,
       the set $\{\mathfrak{m}\in{\rm maxspec(G,u)}\mid\mathfrak{m}\supseteq\mathcal{B}\setminus C \}$
       is homeomorphic to a $(k-1)$-simplex;
    \item[(iii)] there are integers $1 \leq m_{1},\ldots,m_{n}$ such that
      $\sum_{i=1}^{n}m_{i}b_{i} = u$.
  \end{itemize}
\end{definition}

\begin{theorem}\cite[Theorem 3.1]{CM2011}
Let $(G,u)$ be unital $\ell$-group. Then the following are equivalent:
\begin{itemize}
\item[(i)] $(G,u)$ is finitely presented;
\item[(ii)] $(G,u)$ has a basis.
\end{itemize}
\end{theorem}

Given a unital $\ell$-group and  a basis
$\mathcal{B}=\{b_1,\ldots,b_m\}$
 of $(G,u)$, let $\mathfrak{W}_{\mathcal{B}}=\{\mathcal{B},\Sigma_{\mathcal{B}},\omega_{\mathcal{B}}\}$ be the weighted abstract simplicial complex given by the following stipulations:
 \begin{itemize}
 \item[--]
$
  S\in\Sigma_{\mathcal{B}} \mbox{ iff }\bigwedge S\neq 0$
 \item[--]
  $
  \omega_{\mathcal{B}}(b_i)=m_i.$
\end{itemize}

\begin{proposition}\label{Prop:PolBasis}
  Let $(G,u)$ be unital $\ell$-group and $\mathcal B$ be a basis for $(G,u)$.
  Then $$(G,u)\cong \McN(\Pol(\mathfrak{W}_{\mathcal{B}})).$$
  Recall that $\McN(\Pol(\mathfrak{W}_{\mathcal{B}}))$ is the unital $\ell$-group of \Zed-maps from the canonical realization of $\mathfrak{W}_{\mathcal{B}}$ into $\R$.
\end{proposition}

\begin{proof}
This is essentially the content of the proof of
\cite[Theorem 3.1]{CM2011}.
\end{proof}

\begin{theorem}\label{Theo-BasisExact}
  Let $(G,u)$ be a unital $\ell$-group.
  Then $(G,u)$ is exact iff it has a basis $\mathcal{B}=\{b_1,\ldots,b_n\}$
  satisfying the following conditions:
  \begin{itemize}
    \item[(i)] There is an element $b_i\in\mathfrak{B}$, such that $m_i=1$;
    \item[(ii)] For each maximal $S\in \mathfrak{W}_{\mathcal{B}}$
      the greatest common divisor of $\{m_j\mid b_j\in S\}$ is $1$;
    \item[(iii)] For each $b_i,b_j\in \mathcal{B}$
      there exist a sequence $b_i=b_{k_1},b_{k_2},\ldots,b_{k_m}=b_j$,
      such that $b_{k_l}\wedge b_{k_{l+1}}\neq 0$ for each $l\in\{1,\ldots, m-1\}$.
\end{itemize}
\end{theorem}
\begin{proof}
Immediate
 from Theorem \ref{Theo_WeakProj}, Proposition \ref{Prop:PolBasis},
 upon noting that
    \begin{itemize}
     \item[--]  Condition
     (i) is equivalent to $\Pol(\mathfrak{W}_{\mathcal{B}})\cap\mathbb{Z}^{n}\neq\emptyset$;
     \item[--]
     Condition
      (ii) is equivalent to $\Pol(\mathfrak{W}_{\mathcal{B}})$ being strongly regular;
     \item[--]  Condition
     (iii) is equivalent to $\Pol(\mathfrak{W}_{\mathcal{B}})$ being connected.
   \end{itemize}
\end{proof}

\begin{remark}
In \cite[Definition p.3]{Ma20XX}, working in the
framework of Abelian $\ell$-groups,
the author introduced the notion of a
{\em regular} set of positive elements.
This definition only depends
 on the algebraic/combinatorial notions of {\em starrable set} and {\em $1$-regularity}.
In \cite[Lemmas 2.1 and 2.6]{Ma20XX}
it is proved  that,  for  Abelian $\ell$-groups,
regular set of positive generators
coincide with  abstract Schauder bases.
Using this result one can prove that a subset
$\mathcal B$ of a unital $\ell$-group is a basis
iff it is a regular set of positive generators satisfying
condition (iii) in Definition~\ref{def:basis}.
This leads to a reformulation of Theorem~\ref{Theo-BasisExact}
where the exactness of a
unital $\ell$-group is characterized only in terms
of  algebraic-combinatorial notions.
\end{remark}

\subsection{Admissible rules in \L ukasiewicz infinite-valued
calculus}\label{Sec_admissible}
Throughout this paper we have been going back and forth from
unital $\ell$-groups, rational polyhedra and weighted abstract simplicial
complexes.
Using the
categorical equivalence $\Gamma$  between unital $\ell$-groups and
MV-algebras,
the span of our paper can be further extended  the algebraic counterparts of
\L ukasiewicz  infinite-valued  calculus   \L$_\infty$.
This gives us an opportunity to discuss the underlying
algorithmic  and  proof-theoretic aspects of the theory developed so far.
  We refer to \cite{CDM2000} and \cite{Mu2011}  for background on \L$_\infty$
  and MV-algebras, and to \cite{Ry1997} for background on admissible rules.

\medskip
For any set $X$, we will denote ${\sf FORM}_X$ to the algebra of formulas
in the language $\{\top,\neg,\oplus\}$ where $\top$ is a constant
$\neg$ is a unary connective and $\oplus$ is a binary connective and
whose variables are in $X$.
By definition,
a   {\it substitution} $\sigma\colon{\sf FORM}_X\rightarrow \mathsf{FORM}_Y $
is
 a homomorphism of the algebra ${\sf FORM}_X$ into  $\mathsf{FORM}_Y$.

Two formulas $\psi,\varphi$ are said to be {\it equivalent in}
  \L$_\infty$ (in symbols, $\psi\cong_{\text{\L}_\infty}\varphi$) if
   the equation $\psi\approx\varphi$ is valid in every MV-algebra.
   The algebra ${\mathsf{Free}MV}_X={\sf
   FORM}_X/\cong_{\text{\L}_\infty}$ is the free algebra on $X$
   generators in the variety of MV-algebras.

    Let
      $\Gamma$ be the categorical equivalence of \cite[\S 3]{Mu1986}
      between MV-algebras and unital $\ell$-groups.  Then for
any $n$-element set $X$,   ${\mathsf{Free}MV}_X\cong \Gamma(\McN([0,1]^{n}))$.

A {\it rule} in \L$_\infty$ is a pair $(\Theta,\Sigma)$ where
$\Theta\cup \Sigma$ is a finite subset of ${\sf FORM}_X$ for some $X$.
A rule $(\Theta,\{\varphi\})$ is {\it derivable} in \L$_\infty$ if the
quasi-equation
$\bigwedge\{\psi\approx\top\mid\psi\in\Theta\})\rightarrow
\varphi\approx \top$ is valid in every MV-algebra.  A formula
$\varphi$ is a {\it theorem} of \L$_\infty$ if
$(\emptyset,\{\varphi\})$ is derivable in \L$_\infty$.
A rule $(\Theta,\Sigma)$ is said to be {\it admissible} in \L$_\infty$
if for every substitution $\sigma$ such that $\sigma(\psi)$ is a
theorem of \L$_\infty$ for each $\psi\in\Theta$,  then there is
$\varphi\in\Sigma$ such that $\sigma(\varphi)$ is a theorem of
\L$_\infty$.  %In algebraic terms, $(\Theta,\Sigma)$ is admissible if
%and only if the universal formula
%$(\bigwedge\{\psi\approx\top\mid\psi\in\Theta\})\rightarrow
%\bigvee\{\varphi\approx \top\mid\varphi\in\Sigma\}$ is valid in
%${\mathsf{Free}MV}_X$,  for   $X$  the
%set of variables of the formulas $\Theta\cup\Sigma$.

In \cite{Je2010},  the author provides a basis for the
admissible rules
  of \L$_\infty$.
To this purpose, he introduced the notion of admissibly saturated formula.
  An equivalent reformulation is as follows:

\begin{definition}\label{Def:AdmSat}\cite[Definition  3.1]{Je2010} A formula $\varphi$ is {\em
admissibly saturated in \L$_\infty$} if  for every finite set $\Sigma$
of formulas the following conditions are equivalent:
\begin{itemize}
\item[(i)] the rule $(\{\varphi\},\Sigma)$ is admissible in \L$_\infty$;
\item[(ii)] there exists $\psi\in\Sigma$ such that $(\{\varphi\},\{\psi\})$ is derivable in
\L$_\infty$.
\end{itemize}
\end{definition}

A formula $\varphi$ whose set of variables is $X$ is said to be {\it exact} in \L$_\infty$ if there exists a
substitution $\sigma\colon{\sf FORM}_X\rightarrow \mathsf{FORM}_Y $
such that $\sigma(\psi)$ is a theorem of \L$_\infty$ if
$(\{\varphi\},\{\psi\})$ is derivable in \L$_\infty$.
Equivalently, $\varphi$ is exact iff there exists $Y$ such
that
${\mathsf{Free}MV}_X/\theta([\varphi]_{\cong_{\text{\L}_\infty}},[\top]_{\cong_{\text{\L}_\infty}})$
is isomorphic to a subalgebra of ${\mathsf{Free}MV}_Y$, where
$\theta([\varphi]_{\cong_{\text{\L}_\infty}},[\top]_{\cong_{\text{\L}_\infty}})$
denotes the principal congruence generated by
$([\varphi]_{\cong_{\text{\L}_\infty}},[\top]_{\cong_{\text{\L}_\infty}})$.
Using the categorical equivalence
$\Gamma$
 between MV-algebras
and unital $\ell$-groups,  it follows that $\varphi$ is exact iff
${\mathsf{Free}MV}_X/\theta([\varphi]_{\cong_{\text{\L}_\infty}},[\top]_{\cong_{\text{\L}_\infty}})$ is isomorphic to $\Gamma(G,u)$ for some exact unital $\ell$-group.

As is well known,
{\it exact formulas are admissibly
 saturated.}   To help the reader,
  we supply a short proof for the
  special  case of \L$_\infty$.
Let $\varphi $ be an
exact formula whose set of variables is $X$.
Let  $(\{\varphi\},\Sigma)$ be an admissible rule in
\L$_\infty$. By Defintion \ref{Def:AdmSat} we need to prove that there exists
 $\psi\in \Sigma$ such that $(\{\varphi\},\{\psi\})$ is derivable in \L$_\infty$.
Since $\varphi$ is exact, there exist
 a substitution $\sigma\colon{\sf FORM}_X\rightarrow \mathsf{FORM}_Y $
such that $\sigma(\psi)$ is a theorem of \L$_\infty$ and
$(\{\varphi\},\{\psi\})$ is derivable in \L$_\infty$ iff $\sigma(\psi)$ is a theorem of \L$_\infty$.

Since $\sigma(\varphi)$ is a theorem of \L$_\infty$ and $(\{\varphi\},\Sigma)$
is an admissible rule, there exists
 $\psi\in \Sigma$ such that $\sigma(\psi)$ is a
  theorem of \L$_\infty$, i.e.  $(\{\varphi\},\{\psi\})$ is derivable in \L$_\infty$.
This proves that $\varphi$ is admissibly  saturated.

%Let $\theta$ denote the principal congruence of ${\mathsf{Free}MV}_X$ generated by
%$([\varphi]_{\cong_{\text{\L}_\infty}},[\top]_{\cong_{\text{\L}_\infty}})$.
%  Then the universal formula $(\varphi\approx\top)\rightarrow
%\bigvee\{\psi\approx \top\mid\psi\in\Sigma\}$ is valid in
%${\mathsf{Free}MV}_X/\theta$.  Letting $f\colon{\sf FORM}_X
%\rightarrow {\mathsf{Free}MV}_X/\theta $ defined by $\alpha \mapsto
%[[\alpha]_{\cong_{\text{\L}_{\infty}}}]_\theta$, it satisfies
%$f(\varphi)=f(\top)$.
%Therefore, $f(\psi)=f(\top)$ for some
%$\psi\in\Sigma$, that is $[[\alpha]_{\cong_{\text{\L}_{\infty}}}]_\theta=$ is derivable in \L$_\infty$.

\medskip

In Theorem \ref{Thm_combination}
we will prove that exact
and admissibly saturated formulas coincide in \L$_\infty$.

\medskip

The following notion was introduced in \cite[Definition 4.5]{Je2009} to study the decidability of admissible rules in \L ukasiewicz infinite-valued calculus, and used in \cite{Je2010}  to characterize admissibly saturated formulas:
   A set $X\subseteq [0,1]^{n}$ is called a {\it anchored}
    if for some $v_1,\ldots,v_m\in[0,1]^{n}$, $X=\conv(v_1,\ldots,v_m)$ and the affine hull of $X$ intersects
$\Zed^{n}$.

\begin{lemma}\label{Lem_SimplexAnchvsStReg}
   Let $S\subseteq\R^{n}$ be a regular $t$-simplex,
   $t=0,1,\ldots, n$. Then the following conditions are equivalent:
   \begin{itemize}
     \item[(i)] $S$ is strongly regular;
     \item[(ii)] $S$ is anchored;
     \item[(iii)] $S$ is union of finitely many anchored sets.
   \end{itemize}
\end{lemma}
\begin{proof} The equivalence
    (ii)$\Leftrightarrow$(iii)
immediately follows by definition.

   \smallskip

  To prove     (i)$\Leftrightarrow$(ii),
  let  $\{v_0,\ldots,v_t\}$ be the vertices of $S$.
  $(\Rightarrow)$   There exist integers $m_0,\ldots,m_t$ such that
$\sum_{i=0}^{t}m_i\den(v_i)=1$. Since $\den(v_i)v_i\in\Zed^{n}$,
   the affine linear combination $\sum_{i=0}^{t}m_i\den(v_i)v_i$
   lies in $\Zed^n$, whence  $S$ is anchored.

$(\Leftarrow)$ By hypothesis,
  there are $\lambda_0,\ldots,\lambda_t\in \R$ such that
$v=\sum_{i=0}^{t}\lambda_i v_i\in\Zed^{n}$ and
$\sum_{i=0}^{t}\lambda_i =1$. Then
\begin{eqnarray}
\nonumber \tilde{v}&=&\textstyle (\sum_{i=0}^{t}\lambda_i
v_i,1)=(\sum_{i=0}^{t}\lambda_i v_i,\sum_{i=0}^{t}\lambda_i)\\
\nonumber &=&\textstyle
\sum_{i=0}^{t}\lambda_i(v_i,1)=\sum_{i=0}^{t}\frac{\lambda_i}{\den(v_i)}(\den(v_i)(v_i,1))\\
\nonumber &=&\textstyle \sum_{i=0}^{t}\frac{\lambda_i}{\den(v_i)}\tilde{v_i}.
\end{eqnarray}
Since $S$ regular, $\frac{\lambda_i}{\den(v_i)}m_i\in\Zed$.
Therefore, $\sum_{i=0}^{t}m_i\den(v_i)=\sum_{i=0}^{t}\lambda_i=1$,
whence
  the greatest common  divisor
  of $\den(v_0),\ldots,\den(v_t)$ is $1$.
\end{proof}

\begin{theorem}\label{Thm_AnchvsStReg}
   A rational polyhedron $P\subseteq\R^{n}$ is strongly regular iff it is
a finite union of anchored sets.
\end{theorem}
\begin{proof}
$(\Rightarrow)$  Immediate from
  Lemma \ref{Lem_SimplexAnchvsStReg}.

$(\Leftarrow)$
Suppose that $P$ is such that $P=S_1\cup\cdots\cup S_m$ for some anchored sets $S_1,\ldots,S_m\subseteq\R^{n}$

   Let $\Delta$ be a regular triangulation of $P$, and
   $T$  a maximal simplex in $\Delta$, with the intent of proving that
   $T$ is strongly regular.

   Let $v\in{\rm relint}(T)$.
   Since $T\subseteq P \subseteq S_1\cup\cdots\cup S_m$ there exists
$S_i$ such that $v\in S_i$.
Lemma \ref{Lem_MaxSimpTriang}  yields
  $0<\delta_1\in\R$ such that $B(\delta_1,v)\cap P\subseteq T$.
   Since $v\in S_i$, there exists a point
    $w$ in $B(\delta_1,v)\cap{\rm relint}(S_i)$,
whence  $w\in {\rm relint}(T)\cap{\rm relint}(S_i)$.

   Again by Lemma \ref{Lem_MaxSimpTriang},
   there exists $0<\delta_2\in\R$ such that $B(\delta_2,w)\cap P\subseteq T$.
   By definition of the relative interior of $S_i$, there exists
$0<\delta_3\in\R$
   such that $B(\delta_1�,w)\cap {\rm aff}(S_i)\subseteq S_i$.
   Letting $\delta={\rm min}\{\delta_2,\delta_3\}$ we obtain
   \begin{eqnarray}
     \nonumber B(\delta,w)\cap{\rm aff}(S_i)&\subseteq& B(\delta,w)\cap S_i\\
     \nonumber &\subseteq& B(\delta,w)\cap P\\
     \nonumber &\subseteq& T.
   \end{eqnarray}

    Therefore, ${\rm aff}(S_i)\subseteq {\rm aff}(T)$, and  $T$ is anchored.
    Since $T$ is a regular simplex,
    by Lemma \ref{Lem_SimplexAnchvsStReg},
    $T$ is strongly regular, whence so is
  $P$.
\end{proof}

\begin{theorem}\label{Thm_combination}
   Let $\varphi$ be a formula in the language of MV-algebras. Then the
following are equivalent
   \begin{itemize}
    \item[(i)] $\varphi$ is admissibly saturated;
    \item[(ii)] $\varphi$ is exact.
   \end{itemize}
\end{theorem}

\begin{proof}
{}A combination of Theorems \ref{Theo_WeakProj} and
   \ref{Thm_AnchvsStReg} with
  \cite[Theorem 3.5]{Je2010}, again using
  the categorical equivalence  $\Gamma$ of  \cite[Theorem 3.9]{Mu1986}.
\end{proof}

\subsection*{Acknowledgements}

\smallskip
The author would like to thank Professor Daniele Mundici for his
helpful comments and suggestions on previous drafts of this paper.

\end{document}